\documentclass[oneside,english,reqno]{amsart}
\usepackage{lmodern}
\usepackage[T1]{fontenc}
\usepackage[latin9]{inputenc}
\usepackage{babel}
\usepackage{float}
\usepackage{amsthm}
\usepackage{amssymb}
\usepackage[unicode=true,pdfusetitle,
 bookmarks=true,bookmarksnumbered=false,bookmarksopen=false,
 breaklinks=false,pdfborder={0 0 1},backref=false,colorlinks=false]
 {hyperref}
\usepackage{breakurl}

\makeatletter

\floatstyle{ruled}
\newfloat{algorithm}{tbp}{loa}
\providecommand{\algorithmname}{Algorithm}
\floatname{algorithm}{\protect\algorithmname}

\numberwithin{equation}{section}
\numberwithin{figure}{section}
\theoremstyle{plain}
\newtheorem{thm}{\protect\theoremname}[section]
  \theoremstyle{remark}
  \newtheorem{rem}[thm]{\protect\remarkname}
  \theoremstyle{definition}
  \newtheorem{defn}[thm]{\protect\definitionname}
  \theoremstyle{plain}
  \newtheorem{assumption}[thm]{\protect\assumptionname}
  \theoremstyle{plain}
  \newtheorem{lyxalgorithm}[thm]{\protect\algorithmname}
  \theoremstyle{definition}
  \newtheorem{example}[thm]{\protect\examplename}
  \theoremstyle{plain}
  \newtheorem{lem}[thm]{\protect\lemmaname}
  \theoremstyle{plain}
  \newtheorem{prop}[thm]{\protect\propositionname}
  \theoremstyle{plain}
  \newtheorem{cor}[thm]{\protect\corollaryname}

\newcommand{\intr}{\mbox{\rm int}}

\newcommand{\co}{\mbox{\rm co}}
\newcommand{\sspan}{\mbox{\rm span}}

\makeatother

  \providecommand{\algorithmname}{Algorithm}
  \providecommand{\assumptionname}{Assumption}
  \providecommand{\corollaryname}{Corollary}
  \providecommand{\definitionname}{Definition}
  \providecommand{\examplename}{Example}
  \providecommand{\lemmaname}{Lemma}
  \providecommand{\propositionname}{Proposition}
  \providecommand{\remarkname}{Remark}
\providecommand{\theoremname}{Theorem}

\begin{document}
\title[Extended Dykstra's algorithm for polyhedra]{Nonasymptotic and asymptotic linear convergence of an almost cyclic SHQP Dykstra's algorithm for polyhedral problems}

\begin{abstract}
We show that an almost cyclic (or generalized Gauss- Seidel) Dykstra's
algorithm which incorporates the SHQP (supporting halfspace- quadratic
programming) strategy can achieve nonasymptotic and asymptotic linear
convergence for polyhedral problems. 
\end{abstract}

\author{C.H. Jeffrey Pang}

\thanks{We acknowledge grant R-146-000-214-112 from the Faculty of Science,
National University of Singapore.}

\curraddr{Department of Mathematics\\ 
National University of Singapore\\ 
Block S17 08-11\\ 
10 Lower Kent Ridge Road\\ 
Singapore 119076 }

\email{matpchj@nus.edu.sg}

\date{\today{}}


\maketitle
\tableofcontents{}

\section{Introduction}

We consider the following problem, known as the best approximation
problem (BAP).

\begin{eqnarray}
(BAP) & \min & \!\!\!\!\begin{array}{c}
f(x):=\frac{1}{2}\|x-d\|^{2}\end{array}\label{eq:P-primal}\\
 & \mbox{s.t.} & \!\!\!\!\begin{array}{c}
x\in C:=C_{1}\cap\cdots\cap C_{m},\end{array}\nonumber 
\end{eqnarray}
where $d$ is a given point and $C_{i}$, $i=1,\dots,m$, are closed
convex sets in a Hilbert space $X$. The BAP is equivalent to projecting
$d$ onto $C$. We shall assume throughout that $C\neq\emptyset$.
We now give an introduction of the background and techniques of this
paper.

\subsection{Alternating projections and the dual of the BAP}

The BAP is often associated with the set intersection problem (SIP)
\begin{eqnarray}
(SIP) & \mbox{Find } & x\in C:=C_{1}\cap\cdots\cap C_{m}.\label{eq:SIP}
\end{eqnarray}
A well studied method for the SIP is the method of alternating projections
(MAP). We recall material from \cite{BauschkeCombettes11,Deutsch01_survey,Deustch01,EsRa11}
on material on the MAP. As its name suggests, the MAP projects the
iterates in a cyclic or non-cyclic manner so that the iterates converge
to a point in the intersection of these sets. 
\begin{rem}
\label{rem:MAP-lin-subspaces}(MAP on linear subspaces) For future
discussions, we recall that rate of convergence of the MAP when all
the $C_{i}$s are linear subspaces is studied in \cite{Deutsch_Hundal_alt_proj_subspace_rate_JMAA_97},
which builds on the work of \cite{SmithSolmonWagner1977,Kayalar_Weinert}.
See Theorem \ref{thm:DH97} for a corollary of \cite[Theorem 2.7]{Deutsch_Hundal_alt_proj_subspace_rate_JMAA_97}.
\end{rem}
As remarked by several authors, the MAP does not converge to the solution
of the BAP in the general case. Dykstra's algorithm \cite{Dykstra83}
solves the best approximation problem through a sequence of projections
onto each of the sets in a manner similar to the MAP, but correction
vectors are added before every projection. The proof of convergence
to $P_{C}(d)$ was established in \cite{BD86} and sometimes referred
to as the Boyle-Dykstra theorem. For a closed convex set $D\subset X$,
recall that $\delta^{*}(\cdot,D):X\to\mathbb{R}$ is the \emph{support
function} defined by 
\begin{equation}
\delta^{*}(z,D)=\sup_{x\in D}\langle z,x\rangle.\label{eq:def-supp-fn}
\end{equation}
As pointed out in \cite{Han88} and \cite{Gaffke_Mathar}, the dual
problem of the BAP is defined as follows.
\begin{defn}
\label{Def:2-dual-pblms}(Dual problem of the BAP) Let $X$ be a Hilbert
space, $d\in X$, and $C_{i}\subset X$ be closed convex sets such
that $C:=\cap_{i=1}^{m}C_{i}\neq\emptyset$. Following \cite{Han88},
we recall the (Fenchel) dual of the BAP \eqref{eq:P-primal}: 
\begin{eqnarray}
(D') & \underset{y_{1},\dots,y_{m}}{\inf}v(y_{1},\dots,y_{m}),\label{eq:D-prime}
\end{eqnarray}
where $v:X^{m}\to\mathbb{R}$ is defined by 
\begin{equation}
v(y)=v(y_{1},\dots,y_{m})=\underbrace{\frac{1}{2}\left\Vert d-P_{C}(d)-\sum_{i=1}^{m}y_{i}\right\Vert ^{2}}_{(A)}+\underbrace{\sum_{i=1}^{m}\delta^{*}(y_{i},C_{i}-P_{C}(d))}_{(B)},\label{eq:D-prime-d-def}
\end{equation}
where $P_{C}(d)$ denotes the projection of $d$ onto $C$, and $y\in X^{m}$
is the dual variable.
\end{defn}

\subsection{\label{sub:AM}Alternating minimization and variants}

Note that in \eqref{eq:D-prime-d-def}, the underbraced term $(A)$
is smooth, while $(B)$ is a nonsmooth term that is block separable.
The method of alternating minimization (AM) applied to minimizing
\eqref{eq:D-prime-d-def} is to minimize the coordinates $y_{i}$,
$i\in\{1,\dots,m\}$ one at a time in a cyclic manner while holding
all other block coordinates fixed. The papers \cite{Han88} and \cite{Gaffke_Mathar}
also pointed out that Dykstra's algorithm is AM on \eqref{eq:D-prime-d-def}.
AM is also referred to as the block-nonlinear Gauss Seidel method
or block coordinate descent method. 

Since the Hessian of the smooth portion of the subproblem of solving
for one block $y_{i}$ while keeping all other blocks fixed is a multiple
of the identity matrix, the block coordinate (proximal) gradient descent
algorithm (BCGD) in \cite{Tseng_YUn_MAPR_2009,Tseng_Yun_JOTA_2009}
is identical to AM. 

For a matrix $A\in\mathbb{R}^{m\times n}$ and a vector $b\in\mathbb{R}^{m}$
where $m<<n$, the least squares lasso problem is 
\begin{equation}
\begin{array}{c}
\underset{x\in\mathbb{R}^{n}}{\min}\frac{1}{2}\|Ax-b\|^{2}+\lambda\|x\|_{1}.\end{array}\label{eq:lasso}
\end{equation}
The least squares lasso problem is an example of a problem where AM
is a competitive method. A notable but dated paper on applying AM
for this problem is \cite{Saha_Tewari_2013}.

\subsection{\label{sub:Asymp-lin-conv}Asymptotic linear convergence of Dykstra's
algorithm and Alternating Minimization}

We first recall results on the asymptotic linear convergence of Dykstra's
algorithm when the sets $C_{i}$ are all polyhedral.

The first proof of asymptotic linear convergence of a variant of Dykstra's
algorithm was presented in \cite{Iusem_DePierro_Hildreth} for the
case when $C_{i}$ are halfspaces (Dykstra's algorithm coincides with
Hildreth's algorithm for this case). Deutsch and Hundal \cite{Deutsch_Hundal_rate_Dykstra}
refined the linear convergence rate in \cite{Iusem_DePierro_Hildreth}
(also for the case when $C_{i}$ are halfspaces) by applying results
mentioned in Remark \ref{rem:MAP-lin-subspaces}. 

Luo and Tseng \cite{Luo_Tseng_1993_AOR} used a more general framework
to give a different proof of the asymptotic linear convergence of
Dykstra's algorithm when $C_{i}$ are polyhedral. They showed that
if $g:\mathbb{R}^{m}\to\mathbb{R}$ is strongly convex, $E\in\mathbb{R}^{m\times n}$
is a matrix with no zero column, $q\in\mathbb{R}^{n}$ and $X$ is
a polyhedral set, then first order methods (which also includes AM)
applied to 
\begin{equation}
\min_{x\in X}\,g(Ex)+\langle q,x\rangle\label{eq:Luo-Tseng-form}
\end{equation}
has asymptotic linear convergence. (They mentioned that \eqref{eq:D-prime-d-def}
can be transformed into the form \eqref{eq:Luo-Tseng-form}. This
transformation is explicitly stated in \cite{Yun_SIOPT}.) See also
\cite{Tseng_YUn_MAPR_2009}. The proofs in \cite{Luo_Tseng_1993_AOR,Tseng_YUn_MAPR_2009}
are vastly different from that of \cite{Iusem_DePierro_Hildreth,Deutsch_Hundal_rate_Dykstra}. 

The method in \cite{Luo_Tseng_1993_AOR} is superior in some ways
compared to the approach of \cite{Iusem_DePierro_Hildreth,Deutsch_Hundal_rate_Dykstra}.
First, \cite{Luo_Tseng_1993_AOR} allows for multiple coordinates
$y_{i}$ in \eqref{eq:D-prime-d-def} to be minimized at a time instead
of just one coordinate at a time. Secondly, their approach allows
for $C_{i}$ to be polyhedra rather than halfspaces. But the original
approach in \cite{Iusem_DePierro_Hildreth} allows for an almost cyclic
sampling: More precisely, the approach of \cite{Luo_Tseng_1993_AOR}
requires each coordinate to be minimized exactly once in each cycle,
but the approach of \cite{Iusem_DePierro_Hildreth} allows for each
coordinate to be minimized at least once in each cycle instead.

\subsection{\label{sub:Block-alg-nonasymp-rates}Nonasymptotic convergence rates}

Rather than the asymptotic convergence rates, a measure of the effectiveness
of alternating minimization is the nonasymptotic convergence rates
(or absolute rate of convergence). Nonasymptotic rates hold from the
very first iteration, and are more useful than the asymptotic rates
for large scale problems, which can take many iterations to achieve
the asymptotic convergence rates. These rates are typically sublinear,
like $O(1/k)$ for example. A modern elementary reference on the nonasymptotic
convergence of first order methods is \cite{Nesterov_1983}.

The papers \cite{Beck_Tetruashvili_2013,Beck_alt_min_SIOPT_2015}
gave a summary of the history behind AM and showed that AM has an
$O(1/k)$ nonasymptotic rate of convergence for the cases when there
are multiple blocks but no proximal terms (i.e., the term corresponding
to $(B)$ in \eqref{eq:D-prime-d-def} is zero), and when there are
proximal terms but only two blocks. See also \cite{Hong_Wang_Razaviyayn_Luo_MP_BCD}.
For the dual problem corresponding to Dykstra's algorithm, \cite{Chambolle_Pock_2015_acc_BCD}
showed that the techniques in \cite{Beck_Tetruashvili_2013,Beck_alt_min_SIOPT_2015}
give a $O(1/k)$ convergence rate. More can be said for BCGD in general.
For example, \cite{Yun_SIOPT} showed that BCGD has an $O(1/k)$ nonasymptotic
rate of convergence. 

As explained in \cite{Nesterov_1983}, a typical condition needed
for the nonasymptotic linear convergence of first order methods is
the strong convexity of the objective function. Wang and Lin \cite{Wang_Lin_JMLR_2014}
showed that first order methods for problems of the form \eqref{eq:Luo-Tseng-form}
achieve nonasymptotic linear convergence, and \cite{Yun_SIOPT} showed
that AM for \eqref{eq:D-prime-d-def} achieves nonasymptotic linear
convergence.

\subsection{Other notable results on Dykstra's algorithm}

Another aspect of Dykstra's algorithm useful for future discussions
is that Hundal and Deutsch \cite{Hundal-Deutsch-97} showed that Dykstra's
algorithm converges when the sets in Dykstra's algorithm are sampled
in a random order provided that each set is projected onto infinitely
often. (The same paper also showed that Dykstra's algorithm converges
for the case of infinitely many sets, but we will not make use of
this property in this paper.) 

A method studied in \cite{cut_Pang12} and \cite{Pang_DBAP} to improve
convergence of the MAP and Dykstra's algorithm respectively is to
notice that each projection onto a set $C_{i}$ generates a supporting
halfspace of $C_{i}$, which in turn contains $C$, and that the projection
onto the intersection of these halfspaces is relatively easy using
quadratic programming. We call this the SHQP strategy. The SHQP strategy
can be seen as a greedy step, as explained in Remark \ref{rem:SHQP-step}.
For the case when $m$ is small and the $C_{i}$s are halfspaces in
the BAP \eqref{eq:P-primal}, one could apply the SHQP strategy and
solve the BAP in one step.

\subsection{\label{sub:Contributions-of-this-paper}Contributions of this paper}

We provide more context behind our contribution. On the one hand,
the approach in \cite{Iusem_DePierro_Hildreth,Deutsch_Hundal_rate_Dykstra}
gives asymptotic linear convergence for almost cyclic sampling, but
only for $C_{i}$ being halfspaces. On the other hand, the approach
in \cite{Luo_Tseng_1993_AOR,Tseng_YUn_MAPR_2009} give asymptotic
linear convergence for polyhedral problems (i.e., $C_{i}$ can be
any polyhedra), but requires a restricted Gauss-Seidel sampling and
not almost cyclic sampling. It doesn't seem easy to improve the general
strategy in \cite{Luo_Tseng_1993_AOR,Tseng_YUn_MAPR_2009} mentioned
in Subsection \ref{sub:Asymp-lin-conv} to get asymptotic linear convergence.
(In fact, \cite{Tseng_YUn_MAPR_2009} proved that BCGD with almost
cyclic sampling, which they called unrestricted Gauss Seidel, has
global convergence, but they did not address asymptotic linear convergence.) 

Our approach is to build on the techniques of \cite{Iusem_DePierro_Hildreth,Deutsch_Hundal_rate_Dykstra}
together with results in various directions in \cite{Hundal-Deutsch-97,Deutsch_Hundal_alt_proj_subspace_rate_JMAA_97}
to obtain asymptotic linear convergence for almost cyclic sampling
for the case when the sets $C_{i}$ are polyhedral and not just halfspaces.
We also show that we can incorporate the SHQP step and still have
both asymptotic and nonasymptotic linear convergence.

\subsection{Notation}

For integers $l_{1}$ and $l_{2}$ such that $l_{1}\leq l_{2}$, we
write $\{l_{1},l_{1}+1,\dots,l_{2}-1,l_{2}\}$ as $[l_{1},l_{2}]$
in order to simplify notation.

\section{On the least squares lasso}

To further motivate this paper, we first point out a rather elementary
fact that the least squares lasso problem \eqref{eq:lasso} is a special
case of \eqref{eq:D-prime-d-def}, the dual of the BAP, before the
recalling preliminaries for the rest of the paper. 

Recall the lasso problem \eqref{eq:lasso}. Denote the $i$th column
of $A$ to be $A_{i}$. We can assume that none of the $A_{i}$s are
zero since if $A_{i}$ is zero, the $i$th component of any optimal
vector $x$ has to be zero. Consider the following problems \begin{subequations}
\begin{eqnarray}
 &  & \begin{array}{c}
\underset{x\in\mathbb{R}^{n}}{\min}\frac{1}{2}\left\Vert b-\underset{i=1}{\overset{n}{\sum}}A_{i}x_{i}\right\Vert _{2}^{2}+\lambda\|x_{i}\|_{1}\end{array}\label{eq:lasso-1}\\
 &  & \begin{array}{c}
\underset{x\in\mathbb{R}^{n}}{\min}\frac{1}{2}\left\Vert b-\underset{i=1}{\overset{n}{\sum}}\frac{A_{i}}{\|A_{i}\|_{2}}x_{i}\right\Vert _{2}^{2}+\underset{i=1}{\overset{n}{\sum}}\frac{\lambda}{\|A_{i}\|_{2}}|x_{i}|\end{array}\label{eq:lasso-2}\\
 &  & \begin{array}{c}
\underset{(y_{1},\dots,y_{n})\in(\mathbb{R}^{m})^{n}}{\min}\frac{1}{2}\left\Vert b-\underset{i=1}{\overset{n}{\sum}}y_{i}\right\Vert _{2}^{2}+\underset{i=1}{\overset{n}{\sum}}\delta^{*}\left(y_{i},S_{i}\right),\end{array}\label{eq:lasso-4}
\end{eqnarray}
\end{subequations}where the slab $S_{i}\subset\mathbb{R}^{m}$ in
\eqref{eq:lasso-4} is defined by 
\[
\begin{array}{c}
S_{i}=\left\{ z:\left(\frac{A_{i}}{\|A_{i}\|}\right)^{T}z\in\left[-\frac{\lambda}{\|A_{i}\|_{2}},\frac{\lambda}{\|A_{i}\|_{2}}\right]\right\} .\end{array}
\]

The problem \eqref{eq:lasso-1} is equivalent to the least squares
lasso problem in \eqref{eq:lasso}. The problems \eqref{eq:lasso-1}
and \eqref{eq:lasso-2} are equivalent up to a scaling of the coordinates
of $x$. Lastly, we show the equivalence of the problems \eqref{eq:lasso-2}
and \eqref{eq:lasso-4}. If $y_{i}$ is a multiple of $\frac{A_{i}}{\|A_{i}\|}$,
say $y_{i}=\frac{A_{i}}{\|A_{i}\|}x_{i}$, then 
\[
\begin{array}{c}
\delta^{*}(y_{i},S_{i})=\delta^{*}\left(\frac{A_{i}}{\|A_{i}\|}x_{i},S_{i}\right)=\max\left\{ \left(\frac{A_{i}}{\|A_{i}\|}x_{i}\right)^{T}z:z\in S_{i}\right\} =|x_{i}|\frac{\lambda}{\|A_{i}\|}.\end{array}
\]
Next, $y_{i}$ not being a multiple of $\frac{A_{i}}{\|A_{i}\|}$
would mean that $\delta^{*}(y_{i},S_{i})=\infty$. So throughout an
algorithm where the objective value in \eqref{eq:lasso-4} is finite,
the $y_{i}$s are multiples of $\frac{A_{i}}{\|A_{i}\|}$ and identical
to \eqref{eq:lasso-1}. Since \eqref{eq:lasso-4} is of the form \eqref{eq:D-prime-d-def},
we are done.

\section{Preliminaries}

For the BAP \eqref{eq:P-primal}, we point out a few known facts on
the dual function $v(\cdot)$ defined in \eqref{eq:D-prime-d-def}.
\begin{thm}
\label{thm:easy-Dykstra-facts}(Known results on dual functions) It
is known that \eqref{eq:D-prime-d-def} is the dual function of the
BAP. (See for example \cite{Han88,Gaffke_Mathar}.) Suppose $P_{C}(d)=0$.
Then 
\begin{enumerate}
\item For $y\in X^{m}$, let $x:=d-\sum_{i=1}^{m}y_{i}$. Then $\frac{1}{2}\|x-P_{C}(d)\|^{2}\leq v(y)$. 
\item $\inf_{y}v(y)=0$.
\item If $y$ is a minimizer of $v(\cdot)$, then $d-\sum_{i=1}^{m}y_{i}$
is the primal minimizer of the BAP \eqref{eq:P-primal}.
\end{enumerate}
\end{thm}
\begin{proof}
Since $0\in C_{i}-P_{C}(d)$, it follows that $\delta^{*}(y_{i},C_{i}-P_{C}(d))\geq0$
for all $i$, and statement (1) follows. Statement (2) can be obtained
from \cite[pages 32--33]{Gaffke_Mathar}. (For more details on the
elementary steps needed to convert the material in \cite[pages 32--33]{Gaffke_Mathar}
to statement (2), see \cite{Pang_DBAP}.) Note that if $y$ is a minimizer
of $v(\cdot)$, then (1) and (2) imply that $\frac{1}{2}\|d-\sum_{i=1}^{m}y_{i}-P_{C}(d)\|^{2}=0$,
from which we get statement (3). 
\end{proof}
As pointed out in \cite{Han88,Gaffke_Mathar}, Dykstra\textquoteright s
algorithm corresponds to alternating minimization on the dual problem
$(D')$ in \eqref{eq:D-prime}. This detail will be elaborated in
\eqref{eq:derive-e-j}, after we introduce our extended Dykstra's
algorithm. 

We make our assumptions of the polyhedral structure of $C_{i}$ in
\eqref{eq:P-primal}.
\begin{assumption}
\label{assm:poly-C-i}(Polyhedral setting) Let $X$ be a Hilbert space
and let $C_{1}$, $\dots$, $C_{m}$ be $m$ polyhedra in $X$ with
nonempty intersection $C=\cap_{i=1}^{m}C_{i}$. Let $d\in X$. Suppose
that $x_{\infty}:=P_{C}(d)$, and assume without loss of generality
that $x_{\infty}=0$. Suppose each polyhedron $C_{i}$ is defined
by 
\begin{equation}
C_{i}=\cap_{r=1}^{K}\mathcal{H}_{i,r},\label{eq:def-C-i}
\end{equation}
where $\mathcal{H}_{i,r}$ are the halfspaces 
\begin{equation}
\mathcal{H}_{i,r}=\{x\in X:\langle x,f_{i,r}\rangle\leq c_{i,r}\},\label{eq:starting-halfspaces}
\end{equation}
where $f_{i,r}\in X\backslash\{0\}$ and $c_{i,r}\in\mathbb{R}\cup\{\infty\}$.
By scaling, we may assume that $\|f_{i,r}\|=1$. Define the affine
space $H_{i,r}$ to be the boundary of $\mathcal{H}_{i,r}$, i.e.,
\begin{equation}
H_{i,r}=\{x\in X:\langle x,f_{i,r}\rangle=c_{i,r}\}.\label{eq:starting-hyperplanes}
\end{equation}
For each $i\in[1,m]$, consider the polyhedron $C'_{i}$ to be the
set defined similarly to $C_{i}$ such that the halfspaces that are
not tight at $x_{\infty}$ are removed, i.e., 
\[
C'_{i}=\cap\big\{\mathcal{H}_{i,r}:r\in\{1,\dots,K\},x_{\infty}\in H_{i,r}\big\}.
\]
Define $I$ by 
\begin{equation}
I=\{i\in[1,m]:x_{\infty}\in\intr C_{i}\}.\label{eq:def-I}
\end{equation}
(In other words, $I=\{i\in[1,m]:C_{i}'=X\}$.) Assume that for all
$i\notin I$, the first $K'$ halfspaces are tight at $x_{\infty}$,
while the remaining $K-K'$ halfspaces are not tight at $x_{\infty}$.
As a consequence, 
\begin{equation}
C_{i}'=\cap_{r=1}^{K'}\mathcal{H}_{i,r}\mbox{ for all }i\notin I,\label{eq:def-C-primes}
\end{equation}
where the halfspaces $\mathcal{H}_{i,r}$ for $r\in\{1,\dots,K'\}$
are all active at $x_{\infty}$, and the halfspaces $\mathcal{H}_{i,r}$
for $r\in\{K'+1,\dots,K\}$ are all not active at $x_{\infty}$. 
\end{assumption}
When $c_{i,r}=\infty$, then $\mathcal{H}_{i,r}=X$ and $H_{i,r}=\emptyset$.
It is clear to see that Assumption \ref{assm:poly-C-i} do not lose
any generality.

\section{\label{sec:Algorithm-statement}Algorithm statement}

We state our extended Dykstra's algorithm in Algorithm \ref{alg:main-alg}. 

\begin{algorithm}[h]
\begin{lyxalgorithm}
\label{alg:main-alg}(Main algorithm) Suppose $P_{C}(d)=0$. Let $y^{0}\in X^{m}$
be such that $y_{i}^{0}\in X$ are the starting dual variables to
$C_{i}$ for $i\in\{1,\dots,m\}$ (for the dual function \eqref{eq:D-prime-d-def}).
The iterates $y^{k}\in X^{m}$ of the algorithm are such that $d-\sum_{i=1}^{m}y_{i}^{k}$
tries to approximate $P_{C}(d)$. 

01$\quad$For $k=1,2,\dots$

02$\quad$$\quad$Run Algorithm \ref{alg:classical-Dyk} with input
$y^{k-1}$ to get $y^{k}$. 

03$\quad$End for \end{lyxalgorithm}
\end{algorithm}

We now describe the subroutine in Algorithm \ref{alg:classical-Dyk}
using the ideas in \cite{Hundal-Deutsch-97}. See the remarks following
the algorithm for more insight.

\begin{algorithm}
\begin{lyxalgorithm}
\label{alg:classical-Dyk}(One cycle in almost cyclic Dykstra's algorithm
with SHQP) Recall the assumptions stated in Algorithm \ref{alg:main-alg}.

\textbf{Input:} $y^{\circ}\in X^{m}$

\textbf{Output:} $y^{+}\in X^{m}$

01$\quad$Choose $w'$ such that $m\leq w'$.

02$\quad$Define $s:[1-m,w']\to[1,m]$ so that $\cup_{j=1}^{w'}\{s(j)\}=[1,m]$
and 
\begin{equation}
s(i-m)=i\mbox{ for all }i\in[1,m].\label{eq:def-HD-s}
\end{equation}

03$\quad$Define $\pi:[1,w']\times[1,m]\to[1-m,w']$ to be 
\begin{equation}
\pi(j,i)=\max\{j':s(j')=i,j'\leq j\}.\label{eq:def-HD-pi}
\end{equation}

04$\quad$Define $p:[1,w']\to[1-m,w']$ to be 
\begin{equation}
p(j)=\pi(j-1,s(j)).\label{eq:def-HD-p}
\end{equation}

05$\quad$Define $e_{i-m,1}:=y_{i}^{\circ}$ for all $i\in[1,m]$,

06$\quad$Let $x_{0}^{+}\leftarrow d-y_{1}^{\circ}-\cdots-y_{m}^{\circ}$.

07$\quad$For $j=1,2,\dots,w'$

08$\quad$$\quad$$z\leftarrow x_{j-1}^{+}+e_{p(j),j}$ 

09$\quad$$\quad$$x_{j}^{\circ}\leftarrow P_{C_{s(j)}}(z)$ 

10$\quad$$\quad$$e_{j,j}\leftarrow z-x_{j}^{\circ}$

11$\quad$$\quad$SHQP greedy step: 

12$\quad$$\quad$Choose a subset $Q_{j}$ of $\{1,\dots,m\}$. 

13$\quad$$\quad$For all $i\in Q_{j}$, let $P_{i,j}\supset C_{i}$
be polyhedra such that  
\begin{equation}
\delta^{*}(e_{\pi(j,i),j},P_{i,j})=\delta^{*}(e_{\pi(j,i),j},C_{i}).\label{eq:Pij-good-for-SHQP}
\end{equation}

14$\quad$$\quad$Let $\{e_{\pi(j,l),j+1}\}_{l=1}^{m}$ be defined
by 
\begin{eqnarray*}
(e_{\pi(j,1),j+1},\dots,e_{\pi(j,m),j+1})= & \underset{(\tilde{y}_{1},\dots,\tilde{y}_{m})}{\arg\min} & \begin{array}{c}
\frac{1}{2}\left\Vert d-\underset{i=1}{\overset{m}{\sum}}\tilde{y}_{i}\right\Vert ^{2}+\underset{i=1}{\overset{m}{\sum}}\delta^{*}(\tilde{y}_{i},P_{i,j})\end{array}\\
 & \mbox{s.t.} & \begin{array}{c}
\tilde{y}_{i}=e_{\pi(j,i),j}\mbox{ if }i\notin Q_{j}.\end{array}
\end{eqnarray*}
(In other words, only the components in $Q_{j}$ are changed from
before.)

15$\quad$$\quad$$x_{j}^{+}=d-\sum_{i=1}^{m}e_{\pi(j,i),j+1}$. 

16$\quad$End for 

17$\quad$Let the vector $y^{+}\in X^{m}$ be defined by $y_{i}^{+}=e_{\pi(w',i),w'+1}$
for all $i\in[1,m]$.\end{lyxalgorithm}
\end{algorithm}

The SHQP step can be omitted in first reading in order to understand
Algorithm \ref{alg:classical-Dyk}. (That would correspond to the
case when $Q_{j}=\emptyset$ for all $j$.) We now comment on Algorithm
\ref{alg:classical-Dyk}. 
\begin{rem}
(On $s(\cdot)$, $\pi(\cdot,\cdot)$ and $p(\cdot)$) The definitions
of $s(\cdot)$, $\pi(\cdot,\cdot)$ and $p(\cdot)$ come from \cite{Hundal-Deutsch-97}.
For $j\in[1,w']$, the index $s(j)\in[1,m]$ gives the index of the
set being projected onto at the $j$th iteration. Once we substitute
the definition of $\pi(\cdot,\cdot)$ in \eqref{eq:def-HD-pi} onto
the definition of the variable $p(j)$ in \eqref{eq:def-HD-p}, we
see that $p(j)$ is the most recent past index $j'$ for which $s(j)=s(j')$.
To model the original Dykstra's algorithm where the variables are
sampled in a cyclic order, we can set $w'=m$ and $s(i)=i$ for all
$i\in[1,m]$. 
\begin{rem}
\label{rem:warmstart-iterates}(Warmstart solutions) As studied in
\cite{Pang_DBAP}, the definition of $\{e_{i-m,1}\}_{i=1}^{m}$ that
will allow for a warmstart iterate $y^{\circ}\in X^{m}$. The case
$y^{\circ}=0$ reduces to the original Dykstra's algorithm with random
order as explained in \cite{Hundal-Deutsch-97}.
\begin{rem}
\label{rem:known-dyk-ppties}(Known properties of Dykstra's algorithm)
We could have written Algorithm \ref{alg:classical-Dyk} in terms
of the vector $y\in X^{m}$, with this vector $y$ produced at the
$j$th iteration (before the SHQP step) being 
\[
(e_{\pi(j,1),j},e_{\pi(j,2),j},\dots,e_{\pi(j,m),j}).
\]
But the notations $s(\cdot)$, $\pi(\cdot,\cdot)$ and $p(\cdot)$
used in \cite{Hundal-Deutsch-97} and adopted here allow us to reference
intermediate calculations easily. From Algorithm \ref{alg:classical-Dyk},
we have 
\begin{eqnarray}
x_{j}^{+} & \overset{\scriptsize\mbox{Lines 6, 15, Alg. }\ref{alg:classical-Dyk}}{=} & \begin{array}{c}
d-\underset{i=1}{\overset{m}{\sum}}e_{\pi(j,i),j+1}\mbox{ for all }j\in[0,w']\end{array}\label{eq:x-j-from-e}\\
\mbox{ and }x_{j}^{\circ} & \overset{\scriptsize\mbox{Line 10, Alg. }\ref{alg:classical-Dyk}}{=} & \begin{array}{c}
d-\underset{i=1}{\overset{m}{\sum}}e_{\pi(j,i),j}\mbox{ for all }j\in[1,w'].\end{array}\label{eq:x-j-circ-from-e}
\end{eqnarray}
 Furthermore 
\begin{equation}
\begin{array}{c}
x_{0}^{+}\overset{\scriptsize\mbox{Line 6, Alg. }\ref{alg:classical-Dyk}}{=}d-\underset{i=1}{\overset{m}{\sum}}y_{i}^{\circ}\mbox{, and }x_{w'}^{+}\overset{\scriptsize\mbox{Line 17, Alg. }\ref{alg:classical-Dyk}}{=}d-\underset{i=1}{\overset{m}{\sum}}y_{i}^{+}.\end{array}\label{eq:x-j-from-y}
\end{equation}
The variable $e_{j,j}$ can be written as 
\begin{eqnarray}
e_{j,j} & \overset{\scriptsize\mbox{Line 10, Alg. }\ref{alg:classical-Dyk}}{=} & \begin{array}{c}
x_{j-1}^{+}+e_{p(j),j}-P_{C_{s(j)}}(x_{j-1}^{+}+e_{p(j),j})\end{array}\label{eq:derive-e-j}\\
 & = & \begin{array}{c}
\underset{e}{\arg\min}\frac{1}{2}\|x_{j-1}^{+}+e_{p(j),j}-e\|^{2}+\delta^{*}(e,C_{s(j)})\end{array}\nonumber \\
 & \overset{\eqref{eq:x-j-from-e}}{=} & \begin{array}{c}
\underset{e}{\arg\min}\frac{1}{2}\big\| d-\underset{{1\leq i\leq m\atop i\neq s(j)}}{\sum}e_{\pi(j,i),j}-e\big\|^{2}+\delta^{*}(e,C_{s(j)}).\end{array}\nonumber 
\end{eqnarray}
(As is known \cite{Han88,Gaffke_Mathar}, the second equation of \eqref{eq:derive-e-j}
comes from the fact that the optimization problem in the second statement
is the dual of 
\[
\begin{array}{c}
\underset{x}{\min}\frac{1}{2}\|x_{j-1}^{+}+e_{p(j),j}-x\|^{2}+\delta(x,C_{s(j)}),\end{array}
\]
which has primal solution $x=P_{C_{s(j)}}(x_{j-1}^{+}+e_{p(j),j})$
and dual solution $e_{j,j}$.) So recalling the definition of $v(\cdot)$
(see \eqref{eq:D-prime-d-def}) and matching the last formula in \eqref{eq:derive-e-j},
we get the known result that evaluating $e_{j,j}$ corresponds to
minimizing the $s(j)$th coordinate while holding all other coordinates
fixed. (The formula \eqref{eq:derive-e-j} also coincides with the
BCGD algorithm mentioned in Subsection \ref{sub:AM}.) If $w'=m$
and $s(i)=i$ for all $i\in[1,m]$, then such a strategy corresponds
to alternating minimization as discussed in Subsection \ref{sub:AM}.
Hence if $j_{2}=j_{1}+1$, then 
\begin{equation}
v(e_{\pi(j_{2},1),j_{1}},e_{\pi(j_{2},2),j_{1}},\dots,e_{\pi(j_{2},m),j_{1}})\leq v(e_{\pi(j_{1},1),j_{1}},e_{\pi(j_{1},2),j_{1}},\dots,e_{\pi(j_{1},m),j_{1}}).\label{eq:class-dyk-decrease-alg}
\end{equation}
Thus $v(\cdot)$ is nonincreasing as Algorithm \ref{alg:classical-Dyk}
progresses. 
\begin{rem}
\label{rem:SHQP-step}(SHQP step) The supporting halfspace quadratic
programming (SHQP) step in lines 12 to 15 of Algorithm \ref{alg:classical-Dyk}
comes from the observation that the projection onto each set $C_{i}$
performed in line 9 generates a supporting halfspace of the set $C_{i}$,
and that the projection of a point onto the intersection of halfspaces
is a relatively easy problem. See \cite{Pang_DBAP} for more details.
\end{rem}
\end{rem}
\end{rem}
We give two examples motivating the design of Algorithm \ref{alg:classical-Dyk}.\end{rem}
\begin{example}
(Many sets of orthogonal constraints) Consider the problem 
\begin{eqnarray*}
 & \underset{x}{\min} & \!\!\!\!\begin{array}{c}
\frac{1}{2}\|x-d\|^{2}\end{array}\\
 & \mbox{s.t.} & \!\!\!\!\begin{array}{c}
l_{i}\leq A_{i}x\leq u_{i}\mbox{ for }i\in\{1,\dots,m\}.\end{array}
\end{eqnarray*}
Let the $A_{i}\in\mathbb{R}^{m_{i}\times n}$ be such that the rows
of $A_{i}$ are orthonormal. This is the setting of the Algebraic
Reconstruction Technique (ART). (See for example \cite{CensorChenCombettesDavidiHerman12,HermanChen08}.)
Let $C_{i}=\{x:l_{i}\leq A_{i}x\leq u_{i}\}$. Since the rows of $A_{i}$
are orthogonal, the projection onto each $C_{i}$ is equivalent to
the projection onto the $m_{i}$ slabs defined by each row of the
constraint $l_{i}\leq A_{i}x\leq u_{i}$. The supporting halfspace
produced by projecting onto each $C_{i}$ can be used to carry out
the SHQP step.
\begin{example}
(Least squares lasso over multiple random blocks) The least squares
lasso problem in \eqref{eq:lasso} is converted into an equivalent
form in \eqref{eq:lasso-4}. The SHQP step in Algorithm \ref{alg:classical-Dyk}
applied to \eqref{eq:lasso-4} corresponds to minimizing over the
coordinates indexed by $Q_{j}$ in the original lasso problem \eqref{eq:lasso}.
\end{example}
\end{example}

\section{Asymptotic linear convergence 1: Adapting \cite{Iusem_DePierro_Hildreth}}

We present the first proof of asymptotic linear convergence of our
algorithm by adapting the proof of \cite{Iusem_DePierro_Hildreth}. 
\begin{lem}
\label{lem:set-N-1}(Behavior when $v(\cdot)$ sufficiently small)
Suppose Assumption \ref{assm:poly-C-i} holds. Consider Algorithm
\ref{alg:classical-Dyk} with $y^{\circ}\in X^{m}$ as input and $y^{+}\in X^{m}$
as output. We have the following:
\begin{enumerate}
\item [(A)]For all $i\in[1,m]$ and $v\in X$, $\delta^{*}(v,C_{i})=0$
if and only if $v$ lies in the normal cone of $C_{i}$ at $0$. For
$i\in I$ (see \eqref{eq:def-I}), this means that $v=0$, and for
$i\notin I$, it means that $v$ lies in the positive hull of $\{f_{i,r}:r\in[1,K']\}$. 
\end{enumerate}
Moreover, there is an $\bar{\epsilon}>0$ such that if $v(y^{\circ})\leq\bar{\epsilon}$,
then 
\begin{enumerate}
\item For all $j\in\{1,\dots,w'\}$, $P_{C_{s(j)}}(x_{j-1}^{+}+e_{p(j),j})=P_{C'_{s(j)}}(x_{j-1}^{+}+e_{p(j),j})$.
\item For all $j\in\{1,\dots,w'\}$, $\delta^{*}(e_{j,j},C_{s(j)}-P_{C}(d))=0$.

\end{enumerate}
\end{lem}
\begin{proof}
We first prove the first statement in (A). If $v$ lies in the normal
cone of $C_{i}$ at $0$, then when you recall the definition of the
support function $\delta^{*}(\cdot,\cdot)$ in \eqref{eq:def-supp-fn},
we see that $P_{C}(d)=0$ is a maximizer. Thus $\delta^{*}(v,C_{i})=\langle v,0\rangle=0$.
For the converse, suppose $\delta^{*}(v,C_{i})=0$. The definition
of the support function tells us that the halfspace $\{x:\langle v,x\rangle\leq0\}$
contains $C_{i}$. Moreover, $0\in C_{i}$ lies on the boundary of
this halfspace. It follows that $v$ lies in the normal cone of $C_{i}$
at $0$. The second statement in (A) is elementary.

Next, we prove property (1). For each $i\in[1,m]$, we write $C_{i}$
as $C_{i}'\cap\bar{C}_{i}$, where $\bar{C}_{i}$ is the intersection
of the halfspaces defining $C_{i}$ that contain $x_{\infty}$ in
their interior. There is a $\gamma>0$ such that $B(x_{\infty},\gamma)$,
the ball with center $x_{\infty}$ and radius $\gamma$, is contained
in $\bar{C}_{i}$ for all $i\in[1,m]$. Recall that the iterates in
Algorithm \ref{alg:classical-Dyk} give nonincreasing dual objective
values. (See \eqref{eq:class-dyk-decrease-alg}.) Hence if $v(y^{\circ})\leq\bar{\epsilon}$,
then 
\begin{eqnarray*}
 &  & \begin{array}{c}
\frac{1}{2}\left\Vert d-\underset{i=1}{\overset{m}{\sum}}e_{\pi(j,i),j}-P_{C}(d)\right\Vert ^{2}\end{array}\\
 & \overset{\scriptsize\mbox{Thm }\ref{thm:easy-Dykstra-facts}(1)}{\leq} & \begin{array}{c}
v(e_{\pi(j,1),j},e_{\pi(j,2),j},\dots,e_{\pi(j,m),j})\overset{\eqref{eq:class-dyk-decrease-alg}}{\leq}v(y^{\circ})\leq\bar{\epsilon},\end{array}
\end{eqnarray*}
which gives $\left\Vert d-\sum_{i=1}^{m}e_{\pi(j,i),j}\right\Vert ^{2}\leq2\bar{\epsilon}$,
or in other words $\|x_{j}^{\circ}-x_{\infty}\|\leq\sqrt{2\bar{\epsilon}}$
through \eqref{eq:x-j-from-e}. Recall that $x_{j}=P_{C_{s(j)}}(x_{j-1}^{+}+e_{p(j),j})$.
If $\bar{\epsilon}$ is chosen to be such that $\sqrt{2\bar{\epsilon}}\leq\gamma$,
then $\|x_{j}^{\circ}-x_{\infty}\|\leq\sqrt{2\bar{\epsilon}}$ implies
that $x_{j}^{\circ}$ cannot be on the boundaries of the halfspaces
defining $C_{s(j)}$ which contain $x_{\infty}$ in their interior.
Thus property (1) holds. 

To get property (2), first observe that 
\begin{eqnarray}
e_{j,j} & \overset{\scriptsize\mbox{Line 10, Alg \ref{alg:classical-Dyk}}}{=} & (x_{j-1}^{+}+e_{p(j),j})-P_{C_{s(j)}}(x_{j-1}^{+}+e_{p(j),j})\label{eq:derive-supp-fn-value}\\
 & \overset{\scriptsize\mbox{Property }(1)}{=} & (x_{j-1}^{+}+e_{p(j),j})-P_{C'_{s(j)}}(x_{j-1}^{+}+e_{p(j),j}).
\end{eqnarray}
Hence $e_{j,j}$ lies in the normal cone of $C_{s(j)}$ at $0$. We
then apply (A). 
\end{proof}
The following is adapted from \cite[Lemma 3.4]{Deutsch_Hundal_rate_Dykstra}.
(This is similar to \cite[Lemma 2]{Iusem_DePierro_Hildreth}.)
\begin{lem}
\label{lem:compare-lem-3-4}(Adaptation of \cite[Lemma 3.4]{Deutsch_Hundal_rate_Dykstra})
Recall Assumption \ref{assm:poly-C-i}. Suppose 
\[
\begin{array}{c}
\begin{array}{c}
x=d-\underset{i\notin I}{\sum}\underset{r=1}{\overset{K'}{\sum}}\tilde{e}_{i,r},\end{array}\end{array}
\]
where $\tilde{e}_{i,r}=\lambda_{i,r}f_{i,r}$ and $\lambda_{i,r}\geq0$.
There is a $\hat{\epsilon}>0$ such that if $\|x\|\leq\hat{\epsilon}$,
then $x\in L$ and $d\in L$, where 
\begin{eqnarray*}
 &  & L=\sspan\big\{ f_{i,r}:(i,r)\in T\big\},\\
 & \mbox{ and } & T=\{(i,r):i\in[1,m],r\in[1,K'],\lambda_{i,r}>0\}.
\end{eqnarray*}
Furthermore, if $d\neq x_{\infty}$, then $T\neq\emptyset$. \end{lem}
\begin{proof}
We now prove the first part. We have, by the definition of $L$,
\[
\begin{array}{c}
x=d-\underset{i\notin I}{\sum}\underset{r=1}{\overset{K'}{\sum}}\tilde{e}_{i,r}\in d+L.\end{array}
\]
Define 
\[
\hat{\epsilon}:=\min_{F\in\mathcal{F}}d(0,F)=\min_{F\in\mathcal{F}}d(x_{\infty},F),
\]
where $d(p,D)$ is the distance of $p$ to the set $D$ and 
\begin{eqnarray*}
\mathcal{F} & = & \big\{ F=\sspan\{f_{i,j}:(i,j)\in S\}+d:\\
 &  & \phantom{\big\{ F:}\,S\subset\{1,\dots,m\}\times\{1,\dots,K'\}\mbox{ and }d(0,F)>0\big\}.
\end{eqnarray*}
Note that $\hat{\epsilon}$ exists and $\hat{\epsilon}>0$ because
there are only a finite number of subsets of $\{1,\dots,m\}\times\{1,\dots,K'\}$
and $\emptyset\subset\{1,\dots,m\}\times\{1,\dots,K'\}$. By the definition
of $\hat{\epsilon}$, 
\[
\mbox{ either }d(0,d+L)\geq\hat{\epsilon}\mbox{ or }d(0,d+L)=0.
\]
From $x\in d+L$ and $\|x\|<\hat{\epsilon}$, we must have $d(0,d+L)=0$.
That is, $d\in L$. It follows that $x\in L$. This proves the first
statement. 

We now prove the last statement. If $T$ were empty, then $\tilde{e}_{i,r}=0$
for all $i\in\{1,\dots,m\}$ and $r\in\{1,\dots,K'\}$. This would
imply $d=x_{\infty}$, a contradiction.
\end{proof}
We prove a proposition about the SHQP step.
\begin{prop}
\label{prop:Decrease-in-dual}(Decrease in dual function) We have\begin{subequations}\label{eq_m:dual-fn-decrease}
\begin{eqnarray}
v(e_{\pi(j,1),j},\dots,e_{\pi(j,m),j}) & \leq & \!\!\!\!\begin{array}{c}
v(e_{\pi(j-1,1),j},\dots,e_{\pi(j-1,m),j})\end{array}\label{eq:dual-fn-decrease-1}\\
 &  & \!\!\!\!\begin{array}{c}
-\frac{1}{2}\|x_{j}^{\circ}-x_{j-1}^{+}\|^{2}\mbox{ for all }j\in[1,w'],\end{array}\nonumber \\
\mbox{ and }v(e_{\pi(j,1),j+1},\dots,e_{\pi(j,m),j+1}) & \leq & \!\!\!\!\begin{array}{c}
v(e_{\pi(j,1),j},\dots,e_{\pi(j,m),j})\end{array}\label{eq:dual-fn-decrease-2}\\
 &  & \!\!\!\!\begin{array}{c}
-\frac{1}{2}\|x_{j}^{+}-x_{j}^{\circ}\|^{2}\mbox{ for all }j\in[1,w'].\end{array}\nonumber 
\end{eqnarray}
\end{subequations}Suppose further that $\delta^{*}(e_{\pi(j,i),j},C_{i})=\delta^{*}(e_{\pi(j,i),j+1},C_{i})=0$
for all $i\in[1,m]$. Then\begin{subequations}\label{eq_m:simple-dual-form}
\begin{eqnarray}
 &  & \begin{array}{c}
v(e_{\pi(j,1),j},\dots,e_{\pi(j,m),j})=\frac{1}{2}\|x_{j}^{\circ}\|^{2}\mbox{ for all }j\in[1,w'],\end{array}\label{eq:simple-dual-form-1}\\
 & \mbox{ and} & \begin{array}{c}
v(e_{\pi(j,1),j+1},\dots,e_{\pi(j,m),j+1})=\frac{1}{2}\|x_{j}^{+}\|^{2}\mbox{ for all }j\in[0,w'].\end{array}\label{eq:simple-dual-form-2}
\end{eqnarray}
\end{subequations}\end{prop}
\begin{proof}
The formulas in \eqref{eq_m:simple-dual-form} are straightforward
from \eqref{eq:D-prime-d-def} and the assumptions. In view of \eqref{eq:x-j-from-e},
\eqref{eq:dual-fn-decrease-1} is equivalent to 
\begin{eqnarray}
 &  & \begin{array}{c}
v(e_{\pi(j,1),j},\dots,e_{\pi(j,m),j})\end{array}\nonumber \\
 & \leq & \begin{array}{c}
v(e_{\pi(j-1,1),j},\dots,e_{\pi(j-1,m),j})-\frac{1}{2}\|e_{p(j),j}-e_{j,j}\|^{2}.\end{array}\label{eq:dual-fn-decrease-long-form}
\end{eqnarray}
We now show that inequality \eqref{eq:dual-fn-decrease-long-form}
holds. Let $i^{*}=s(j)$, which implies $\pi(j,i^{*})=j$. We also
note that $e_{\pi(j,i),j}=e_{\pi(j-1,i),j}$ if $i\neq i^{*}$. Then
\eqref{eq:dual-fn-decrease-long-form} can be written as
\[
\begin{array}{rcl}
 &  & \frac{1}{2}\|d-e_{j,j}-\underset{{1\leq i\leq m\atop i\neq i^{*}}}{\sum}e_{\pi(j,i),j}\|^{2}+\delta^{*}(e_{j,j},C_{i^{*}})\\
 & \leq & \frac{1}{2}\|d-e_{p(j),j}-\underset{{1\leq i\leq m\atop i\neq i^{*}}}{\sum}e_{\pi(j,i),j}\|^{2}+\delta^{*}(e_{p(j),j},C_{i^{*}})-\frac{1}{2}\|e_{p(j),j}-e_{j,j}\|^{2}.
\end{array}
\]
Since $e_{j,j}$ is the minimizer to the function 
\[
\begin{array}{c}
e\mapsto\frac{1}{2}\|d-e-\underset{{1\leq i\leq m\atop i\neq i^{*}}}{\sum}e_{\pi(j,i),j}\|^{2}+\delta^{*}(e,C_{i^{*}}),\end{array}
\]
(see remark \ref{rem:known-dyk-ppties}), which is strongly convex
with modulus 1, we see that \eqref{eq:dual-fn-decrease-long-form}
holds. To prove that \eqref{eq:dual-fn-decrease-2} holds, we look
at the following chain of inequalities:
\begin{equation}
\begin{array}{rcl}
 &  & \frac{1}{2}\|d-\underset{{1\leq i\leq m\atop i\in Q_{j}}}{\sum}e_{\pi(j,i),j+1}-\underset{{1\leq i\leq m\atop i\notin Q_{j}}}{\sum}e_{\pi(j,i),j+1}\|^{2}+\underset{{1\leq i\leq m\atop i\in Q_{j}}}{\sum}\delta^{*}(e_{\pi(j,i),j+1},C_{i})\\
 & \leq & \frac{1}{2}\|d-\underset{{1\leq i\leq m\atop i\in Q_{j}}}{\sum}e_{\pi(j,i),j+1}-\underset{{1\leq i\leq m\atop i\notin Q_{j}}}{\sum}e_{\pi(j,i),j+1}\|^{2}+\underset{{1\leq i\leq m\atop i\in Q_{j}}}{\sum}\delta^{*}(e_{\pi(j,i),j+1},P_{i,j})\\
 & \leq & \frac{1}{2}\|d-\underset{{1\leq i\leq m\atop i\in Q_{j}}}{\sum}e_{\pi(j,i),j}-\underset{{1\leq i\leq m\atop i\notin Q_{j}}}{\sum}e_{\pi(j,i),j}\|^{2}+\underset{{1\leq i\leq m\atop i\in Q_{j}}}{\sum}\delta^{*}(e_{\pi(j,i),j},P_{i,j})\\
 &  & -\frac{1}{2}\|\underset{{1\leq i\leq m\atop i\in Q_{j}}}{\sum}(e_{\pi(j,i),j}-e_{\pi(j,i),j+1})\|^{2}\\
 & \overset{\eqref{eq:Pij-good-for-SHQP}}{=} & \frac{1}{2}\|d-\underset{{1\leq i\leq m\atop i\in Q_{j}}}{\sum}e_{\pi(j,i),j}-\underset{{1\leq i\leq m\atop i\notin Q_{j}}}{\sum}e_{\pi(j,i),j}\|^{2}+\underset{{1\leq i\leq m\atop i\in Q_{j}}}{\sum}\delta^{*}(e_{\pi(j,i),j},C_{i})\\
 &  & -\frac{1}{2}\|\underset{{1\leq i\leq m\atop i\in Q_{j}}}{\sum}(e_{\pi(j,i),j}-e_{\pi(j,i),j+1})\|^{2}.
\end{array}\label{eq:long-chain-str-cvx-decrease}
\end{equation}
The first inequality holds because $P_{i,j}\supset C_{i}$ implies
that $\delta^{*}(\cdot,P_{i,j})\geq\delta^{*}(\cdot,C_{i})$. The
second inequality holds because the variables $\{e_{\pi(j,i),j+1}\}_{i\in Q_{j}}$
are the minimizers of a block coordinate minimization problem whose
smooth function is quadratic. Notice that by the definition of $x_{j}^{+}$
and $x_{j}^{\circ}$ from \eqref{eq:x-j-from-e} and \eqref{eq:x-j-circ-from-e}
that 
\begin{equation}
\begin{array}{c}
\underset{{1\leq i\leq m\atop i\in Q_{j}}}{\sum}(e_{\pi(j,i),j+1}-e_{\pi(j,i),j})=x_{j}^{+}-x_{j}^{\circ}.\end{array}\label{eq:x-j-circ-plus-diff}
\end{equation}
Combining \eqref{eq:long-chain-str-cvx-decrease} and \eqref{eq:x-j-circ-plus-diff}
gives \eqref{eq:dual-fn-decrease-2}.
\end{proof}
The following result is immediate from Proposition \ref{prop:Decrease-in-dual}.
\begin{cor}
\label{cor:decrease-after-w-steps}Recall the iterates of Algorithm
\ref{alg:classical-Dyk}. If $\delta^{*}(e_{\pi(j,i),j},C_{i})=\delta^{*}(e_{\pi(j,i),j+1},C_{i})=0$
for all $i\in[1,m]$ and $j\in[1,w']$, then 
\[
\begin{array}{c}
\|x_{w'}^{+}-0\|^{2}\leq\|x_{0}^{+}-0\|^{2}-\underset{j=1}{\overset{w'}{\sum}}[\|x_{j}^{+}-x_{j}^{\circ}\|^{2}+\|x_{j}^{\circ}-x_{j-1}^{+}\|^{2}].\end{array}
\]
\end{cor}
\begin{proof}
Sum up the terms in \eqref{eq_m:dual-fn-decrease}. Recalling the
definition of $v(\cdot)$ in \eqref{eq:D-prime-d-def}, substitute
in \eqref{eq_m:simple-dual-form} and multiply by 2.
\end{proof}
Recall the definition $I$ in \eqref{eq:def-I}. We make the following
definitions 
\begin{eqnarray}
\bar{A} & := & \{(i,r)\in[1,m]\times[1,K']:\langle f_{i,r},0\rangle=c_{i,r}\}=([1,m]\backslash I)\times[1,K'],\qquad\label{eq:def-A}\\
L_{\bar{A}} & := & \sspan\{f_{i,r}:(i,r)\in\bar{A}\}.\nonumber 
\end{eqnarray}
We thus have $L_{\bar{A}}^{\perp}=\{x:\langle f_{i,r},x\rangle=0\mbox{ for all }(i,r)\in\bar{A}\}=\cap_{(i,r)\in\bar{A}}H_{i,r}$.
The following result is well known. (For example, \cite{Iusem_DePierro_Hildreth}
cited \cite{Goffin80}.)
\begin{lem}
(Regularity in system of equations) \label{lem:lin-reg}There is a
constant $\bar{\mu}>0$ such that for all $x$, $\max_{(i,r)\in\bar{A}}d(x,H_{i,r})>\bar{\mu}d(x,L_{\bar{A}}^{\perp})$. 
\end{lem}
We continue with the proof of the asymptotic linear convergence.
\begin{thm}
\label{thm:IP90-lin-conv}(Asymptotic linear convergence) Recall the
conditions in Lemmas \ref{lem:set-N-1} and \ref{lem:compare-lem-3-4}
and the constant $\mu$ from Lemma \ref{lem:lin-reg}. Suppose in
Algorithm \ref{alg:classical-Dyk}, we have $\|x_{0}^{+}\|\leq\min(\hat{\epsilon},\sqrt{2\bar{\epsilon}})$
and $\delta^{*}(y_{i}^{\circ},C_{i})=0$ for all $i$. Suppose further
that 
\begin{enumerate}
\item For all $i\in[1,m]$, $j\geq\pi(w',i)$ implies that the polyhedron
$P_{i,j}$ is the halfspace $P_{i,j}=\{x:\langle e_{\pi(w',i),\pi(w',i)},x\rangle\leq0\}$.
\item For all $i\in[1,m]$, $j<\pi(w',i)$ implies that the polyhedron $P_{i,j}$
is the intersection of halfspaces with $0$ on their boundaries.
\end{enumerate}
Then we have 
\begin{equation}
\begin{array}{c}
\sqrt{1+\frac{\mu}{2w'}}\|x_{w'}^{+}-0\|\leq\|x_{0}^{+}-0\|.\end{array}\label{eq:lin-conv-IP90-concl}
\end{equation}
\end{thm}
\begin{proof}
Our assumptions ensure that we can apply the conclusions of those
lemmas. Consider the point $x_{w'}^{+}$, which can be written as
$x_{w'}^{+}=d-\sum_{i\notin I}e_{\pi(w',i),w'+1}$ by making use of
\eqref{eq:x-j-from-e} and Lemma \ref{lem:set-N-1}.

From Lemma \ref{lem:set-N-1}, the condition that $\delta^{*}(y_{i}^{\circ},C_{i})=0$
for all $i$, and conditions (1) and (2), we see that $\delta^{*}(e_{j,k},C_{s(j)})=0$
for all $k\in[1,w'+1]$ and $j\leq k$. Moreover, for $j\in[1,w']$
and $k\in[j,w'+1]$, we can write 
\begin{equation}
\begin{array}{c}
e_{j,k}=\underset{r=1}{\overset{K'}{\sum}}\tilde{e}_{j,r,k},\end{array}\label{eq:decomp-e}
\end{equation}
where $\tilde{e}_{j,r,k}=f_{s(j),r}\lambda_{j,r,k}$ for some $\lambda_{j,r,k}\geq0$.
We can assume that $\tilde{e}_{j,r,j}$ are chosen so that they are
the multipliers to the projection step $x_{j}^{\circ}=P_{C_{s(j)}}(z)$
in line 9 of Algorithm \ref{alg:classical-Dyk}, which would give
\begin{equation}
\tilde{e}_{j,r,j}\neq0\mbox{ implies }x_{j}^{\circ}\in H_{s(j),r}.\label{eq:KKT-mult-from-prj}
\end{equation}
 So 
\begin{equation}
\begin{array}{c}
x_{w'}^{+}\overset{\eqref{eq:x-j-from-e},\eqref{eq:decomp-e}}{=}d-\underset{i\notin I}{\sum}\underset{r=1}{\overset{K'}{\sum}}\tilde{e}_{\pi(w',i),r,w'+1}.\end{array}\label{eq:use-decomp-e}
\end{equation}
Let $A=\{(i,r):\tilde{e}_{\pi(w',i),r,w'+1}\neq0\}$, and let $L_{A}=\sspan\{f_{i,r}:(i,r)\in A\}$.
It is clear that $x_{w'}^{+}\in d+L_{A}$. By Lemma \ref{lem:set-N-1},
we have $A\subset\bar{A}$. By Lemma \ref{lem:compare-lem-3-4}, we
have $x_{w'}^{+}\in L_{A}$. Since $0\in L_{A}^{\perp}$ and $x_{w'}^{+}-0\in L_{A}$,
we have $P_{L_{A}^{\perp}}(x_{w'}^{+})=0$. In other words, 
\begin{equation}
\|x_{w'}^{+}-0\|=d(x_{w'}^{+},L_{A}^{\perp}).\label{eq:distance-last-iter-1}
\end{equation}
By Lemma \ref{lem:lin-reg}, there is a $\mu>0$ such that $d(x,L_{A}^{\perp})\leq\frac{1}{\mu}\max_{(i,r)\in A}d(x,H_{i,r})$
for all $x$. (The $\mu>0$ can be chosen to be independent of $A$
by taking the infimum over all $A\subset\bar{A}$.) Thus 
\begin{equation}
\begin{array}{c}
d(x_{w'}^{+},L_{A}^{\perp})\leq\frac{1}{\mu}\underset{(i,r)\in A}{\max}d(x_{w'}^{+},H_{i,r}).\end{array}\label{eq:distance-last-iter-2}
\end{equation}
Let $H_{\bar{i},\bar{r}}$ be a hyperplane such that the maximum
in \eqref{eq:distance-last-iter-2} is attained. By the condition
(1) in the theorem statement, the vectors in $\{e_{\pi(w',\bar{i}),j}\}_{j=\pi(w',\bar{i})+1}^{w'+1}$
are all multiples of $e_{\pi(w',\bar{i}),\pi(w',\bar{i})}$. Since
$(\bar{i},\bar{r})\in A$, we have (from the definition of $A$) that
$\tilde{e}_{\pi(w',\bar{i}),\bar{r},w'+1}\neq0$, which implies $\tilde{e}_{\pi(w',\bar{i}),\bar{r},\pi(w',\bar{i})}\neq0$.
Therefore, the point $x_{\pi(w',\bar{i})}^{\circ}$ lies in the hyperplane
$H_{\bar{i},\bar{r}}$ by \eqref{eq:KKT-mult-from-prj}. The usual
triangular inequality implies that 
\begin{eqnarray}
d(x_{w'}^{+},H_{\bar{i},\bar{r}}) & \leq & \begin{array}{c}
\|x_{\pi(w',\bar{i})}^{+}-x_{\pi(w',\bar{i})}^{\circ}\|\end{array}\label{eq:distance-last-iter-3}\\
 &  & \begin{array}{c}
+\underset{j=\pi(w',\bar{i})+1}{\overset{w'}{\sum}}[\|x_{j}^{+}-x_{j}^{\circ}\|+\|x_{j}^{\circ}-x_{j-1}^{+}\|]\end{array}\nonumber \\
 & \leq & \begin{array}{c}
\underset{j=1}{\overset{w'}{\sum}}[\|x_{j}^{+}-x_{j}^{\circ}\|+\|x_{j}^{\circ}-x_{j-1}^{+}\|].\end{array}\nonumber 
\end{eqnarray}
The Cauchy Schwarz inequality gives 
\begin{eqnarray}
\|x_{w'}^{+}-0\|^{2} & \overset{\eqref{eq:distance-last-iter-1},\eqref{eq:distance-last-iter-2},\eqref{eq:distance-last-iter-3}}{\leq} & \begin{array}{c}
\frac{1}{\mu^{2}}\Big(\underset{j=1}{\overset{w'}{\sum}}[\|x_{j}^{+}-x_{j}^{\circ}\|+\|x_{j}^{\circ}-x_{j-1}^{+}\|]\Big)^{2}\end{array}\label{eq:distance-last-iter-4}\\
 & \leq & \begin{array}{c}
\frac{2w'}{\mu^{2}}\underset{j=1}{\overset{w'}{\sum}}[\|x_{j}^{+}-x_{j}^{\circ}\|^{2}+\|x_{j}^{\circ}-x_{j-1}^{+}\|^{2}].\end{array}\nonumber 
\end{eqnarray}
Therefore, 
\begin{eqnarray*}
\|x_{w'}^{+}-0\|^{2} & \overset{\scriptsize\mbox{Cor. }\ref{cor:decrease-after-w-steps}}{\leq} & \begin{array}{c}
\|x_{0}^{+}-0\|^{2}-\underset{j=1}{\overset{w'}{\sum}}[\|x_{j}^{+}-x_{j}^{\circ}\|^{2}+\|x_{j}^{\circ}-x_{j-1}^{+}\|^{2}]\end{array}\\
 & \overset{\eqref{eq:distance-last-iter-4}}{\leq} & \begin{array}{c}
\|x_{0}^{+}-0\|^{2}-\frac{\mu^{2}}{2w'}\|x_{w'}^{+}-0\|^{2}.\end{array}
\end{eqnarray*}
Rearranging the above gives us \eqref{eq:lin-conv-IP90-concl} as
needed. 
\end{proof}

\section{\label{sec:Polyhedral-lin-conv}Asymptotic linear convergence 2:
Adapting \cite{Deutsch_Hundal_rate_Dykstra}}

In this section, we show that the iterations in Algorithm \ref{alg:classical-Dyk}
result in an asymptotic linear convergence of the primal objective
value when the sets $C_{i}$ are polyhedral.

\subsection{Preliminaries and results from \cite{Deutsch_Hundal_rate_Dykstra}}

In this subsection, we list the assumptions we make, and also recall
some results from \cite{Deutsch_Hundal_rate_Dykstra} useful for the
proof of asymptotic linear convergence. 
\begin{lem}
\label{lem:co-P-commute}(See \cite[Lemma 3.2]{Deutsch_Hundal_rate_Dykstra})
If $H$ is a closed linear variety in a Hilbert space $X$, then $P_{H}(\cdot)$
is ``affine'', that is, 
\[
\begin{array}{c}
P_{H}\bigg(\underset{i=1}{\overset{n}{\sum}}\alpha_{i}x_{i}\bigg)=\underset{i=1}{\overset{n}{\sum}}\alpha_{i}P_{H}(x_{i})\end{array}
\]
 for all $x_{i}\in X$ and any $\alpha_{i}\in\mathbb{R}$ which satisfy
$\sum_{i=1}^{n}\alpha_{i}=1$. 

In particular, if $A$ is any nonempty subset of $X$, then 
\[
P_{H}(\co(A))=\co(P_{H}(A)).
\]

\end{lem}
We have the following result that was proved within bigger results
in \cite{Deutsch_Hundal_rate_Dykstra}.
\begin{lem}
\label{lem:From-thm-2-4}(Local behavior of Dykstra-like iterations)
Let $X$ be a Hilbert space. Let $\mathcal{H}\subset X$ be the halfspace
$\{x:\langle x,f\rangle\leq0\}$, where $\|f\|=1$, and let $H$ be
the hyperplane $\{x:\langle x,f\rangle=0\}$. 
\begin{enumerate}
\item For any $x\in X$ and $\lambda\geq0$, we have $P_{\mathcal{H}}(x+\lambda f)\in\co\{x,P_{H}(x)\}.$
\item If $x+\lambda f-P_{\mathcal{H}}(x+\lambda f)\neq0$, then $P_{\mathcal{H}}(x+\lambda f)=P_{H}(x+\lambda f)$.
\end{enumerate}
\end{lem}
\begin{proof}
We can easily check that $P_{H}(x+\lambda f)=P_{H}(x)$. If $P_{\mathcal{H}}(x+\lambda f)=P_{H}(x)$,
then (1) holds. If $P_{\mathcal{H}}(x+\lambda f)\neq P_{H}(x+\lambda f)=P_{H}(x)$,
then the only possibility is that $x+\lambda f\in\intr\mathcal{H}$,
so that $P_{\mathcal{H}}(x+\lambda f)=x+\lambda f$. The conclusion
(1) can also be easily checked. Conclusion (2) is also easy to check.
\end{proof}

\subsection{Proof of result}

In this subsection, we present the asymptotic linear convergence result
and its proof.

We have the following generalization of \cite[Theorem 3.3]{Deutsch_Hundal_rate_Dykstra}. 
\begin{thm}
\label{thm:From-thm-3-3}(An estimate of iterates in Dykstra's algorithm)
Let $m_{1}$ and $m_{2}$ be two integers. For $i=1,2$, define $H_{j}^{(i)}$,
where $j\in\{1,\dots,m_{i}\}$, are subspaces of a Hilbert space $X$.
Define sets $E^{(i)}\in\{0,1\}^{m_{i}}$ and $D^{(i)}\in\{0,1\}^{m_{i}}$
such that $E^{(i)}\geq D^{(i)}$. Let $x_{1}$, $x_{2}$ and $x_{3}$
be points in $X$. Define $K(E^{(i)},D^{(i)},x_{i})$ by 
\begin{equation}
K(E^{(i)},D^{(i)},x_{i}):=\co\big\{\bar{P}_{S^{(i)}}(x_{i}):S^{(i)}\in\{0,1\}^{m_{i}}\mbox{ and }E^{(i)}\geq S^{(i)}\geq D^{(i)}\big\},\label{eq:the-K}
\end{equation}
where 
\begin{eqnarray}
\bar{P}_{S^{(i)}}(x_{i}) & = & Q_{S^{(i)},m_{i}}Q_{S^{(i)},m_{i}-1}\cdots Q_{S^{(i)},1}(x_{i}),\label{eq:def-bar-P}\\
\mbox{ and }Q_{S^{(i)},j} & = & \begin{cases}
P_{H_{j}^{(i)}} & \mbox{ if }S_{j}^{(i)}=1\\
I & \mbox{ if }S_{j}^{(i)}=0.
\end{cases}\nonumber 
\end{eqnarray}
If $x_{i+1}\in K(E^{(i)},D^{(i)},x_{i})$ for $i\in\{1,2\}$, then
$x_{3}\in K([E^{(1)},E^{(2)}],[D^{(1)},D^{(2)}],x_{1})$, where $[E^{(1)},E^{(2)}]$
and $[D^{(1)},D^{(2)}]$ are the concatenation of the respective vectors,
and the hyperplanes are relabeled accordingly. \end{thm}
\begin{proof}
Note from the definition of $K(\cdot,\cdot,\cdot)$ that our result
would hold if we can prove that\begin{subequations}\label{eq:key-to-3-3}
\begin{eqnarray}
\!\!\!\!\!\!\!\! &  & \co\big\{\bar{P}_{S^{(2)}}\big(\co\{\bar{P}_{S^{(1)}}(x_{1}):E^{(1)}\geq S^{(1)}\geq D^{(1)}\}\big):E^{(2)}\geq S^{(2)}\geq D^{(2)}\big\}\label{eq:key-to-3-3-1}\\
\!\!\!\!\!\!\!\! & \subset & \co\{\bar{P}_{S^{(2)}}\bar{P}_{S^{(1)}}(x_{1}):E^{(i)}\geq S^{(i)}\geq D^{(i)}\mbox{ for }i=1,2\}.\label{eq:key-to-3-3-2}
\end{eqnarray}
\end{subequations}We write down a claim whose proof is embedded within
its statement. 

\textbf{Claim:} Suppose $P_{1}$ and $P_{2}$ are affine operators
in the sense of Lemma \ref{lem:co-P-commute}. Let $A$ be any set
in $X$. From the fact that $P_{2}(P_{1}(x))=(P_{1}P_{2})(x)$, we
have 
\begin{eqnarray}
(P_{2}P_{1})\big(\co(A)\big) & = & P_{2}\Big(P_{1}\big(\co(A)\big)\Big)\label{eq:recur-for-3-3}\\
 & \overset{\scriptsize\mbox{Lem. }\ref{lem:co-P-commute}}{=} & P_{2}\Big(\co\big(P_{1}(A)\big)\Big)\nonumber \\
 & \overset{\scriptsize\mbox{Lem. }\ref{lem:co-P-commute}}{=} & \co\Big(P_{2}\big(P_{1}(A)\big)\Big)=\co\big((P_{2}P_{1})(A)\big).\nonumber 
\end{eqnarray}
By making use of the principle in \eqref{eq:recur-for-3-3}, the term
in \eqref{eq:key-to-3-3-1} can be seen to be 
\begin{equation}
\co\big\{\co\{\bar{P}_{S^{(2)}}\bar{P}_{S^{(1)}}(x_{1}):E^{(1)}\geq S^{(1)}\geq D^{(1)}\}:E^{(2)}\geq S^{(2)}\geq D^{(2)}\big\}.\label{eq:key-t-3-3-3}
\end{equation}
To prove \eqref{eq:key-t-3-3-3}$\subset$\eqref{eq:key-to-3-3-2},
note that this inclusion can be phrased as 
\begin{equation}
\co\big\{\co\{p_{i,j}:i\in I^{*}\}:j\in J^{*}\big\}\subset\co\{p_{i,j}:i\in I^{*},j\in J^{*}\},\label{eq:rephrased-for-3-3}
\end{equation}
where $I^{*}$ and $J^{*}$ are two index sets and $p_{i,j}\in X$
corresponds to $\bar{P}_{S^{(2)}}\bar{P}_{S^{(1)}}(x_{1})$. It is
clear every element on the left hand side of \eqref{eq:rephrased-for-3-3}
can be written as a convex combination of the $p_{i,j}$s, so \eqref{eq:rephrased-for-3-3}
holds. Thus we are done.
\end{proof}

\begin{prop}
\label{prop:warmstart-Dykstra-to-find-e-j}(Dual problem in breaking
up $C_{i}$) Recall Assumption \ref{assm:poly-C-i} and Algorithm
\ref{alg:classical-Dyk}. Suppose that $s(j)=i$. Recall from Remark
\ref{rem:known-dyk-ppties} that 
\begin{equation}
e_{j,j}=\underset{e}{\arg\min}\bigg[\frac{1}{2}\Big\|\underbrace{d-\underset{{1\leq i'\leq m\atop i'\neq i}}{\sum}e_{\pi(j,i'),j}}_{x_{j-1}^{+}+e_{p(j),j}}-e\Big\|^{2}+\delta^{*}(e,C_{i})\bigg].\label{eq:e-j-j-large-min}
\end{equation}
Recall also that $C_{i}=\cap_{r=1}^{K}\mathcal{H}_{i,r}$. Define
the function $v':X^{K}\to\mathbb{R}\cup\{\infty\}$ by 
\begin{eqnarray}
v'(y_{1}'',\dots,y_{K}'') & := & \!\!\!\!\begin{array}{c}
\frac{1}{2}\big\| x_{j-1}^{+}+e_{p(j),j}-\underset{r'=1}{\overset{K}{\sum}}y_{r'}''\big\|^{2}+\underset{r'=1}{\overset{K}{\sum}}\delta^{*}(y_{r'}'',\mathcal{H}_{i,r'}).\quad\end{array}\label{eq:defn-v-prime}
\end{eqnarray}
If $y_{r'}''$ are chosen such that 
\begin{equation}
\begin{array}{c}
\underset{r'=1}{\overset{K}{\sum}}y_{r'}''=e_{p(j),j}\mbox{, and }\underset{r'=1}{\overset{K}{\sum}}\delta^{*}(y_{r'}'',\mathcal{H}_{i,r'})=\delta^{*}(e_{p(j),j},C_{i}),\end{array}\label{eq:decompose-y-pp}
\end{equation}
 then 
\begin{equation}
\begin{array}{c}
v'(y_{1}'',\dots,y_{K}'')+\underset{{1\leq i'\leq m\atop i'\neq i}}{\sum}\delta^{*}(e_{\pi(j,i'),j},C_{i'})=v(e_{\pi(j,1),j},\dots,e_{\pi(j,m),j}).\end{array}\label{eq:expanded-polyhedron-and-dual-val}
\end{equation}
Moreover, let $(y_{1}',\dots,y_{K}')$ be such that 
\begin{equation}
(y_{1}',\dots,y_{K}')\in\underset{(y_{1}'',\dots,y_{K}'')}{\arg\min}v'(y_{1}'',\dots,y_{K}'').\label{eq:def-Dykstra-subpblm}
\end{equation}
 Then the term $e_{j,j}$ is equal to $e_{j,j}=\sum_{r=1}^{K}y'_{r}$.\end{prop}
\begin{proof}
The equality \eqref{eq:expanded-polyhedron-and-dual-val} follows
directly from \eqref{eq:decompose-y-pp} and \eqref{eq:e-j-j-large-min},
and how $v(\cdot)$ and $v'(\cdot)$ are defined in \eqref{eq:D-prime-d-def}
and \eqref{eq:defn-v-prime}. The formula for $e_{j,j}$ follows from
the background theory of Dykstra's algorithm. 
\end{proof}
We recall a warmstart Dykstra's algorithm for finding $e_{j,j}$ from
$e_{p(j),j}$. 
\begin{lyxalgorithm}
\label{alg:small-Dykstra}(Warmstart Dykstra's algorithm) Consider
the problem of finding $e_{j,j}$ from $e_{p(j),j}$ in lines 8 to
10 of Algorithm \ref{alg:classical-Dyk}, and suppose $s(j)=i$. Let
$e_{p(j),j}=\sum_{r=1}^{K'}\lambda_{r-K'}f_{i,r}$, where $\lambda_{r}\geq0$
for all $r\in[1-K',0]$. The sequence $\{\lambda_{r}\}_{r=1-K'}^{\infty}$
is defined as follows.

01$\quad$$x'_{0}=x_{j-1}^{+}$

02$\quad$For $t=1,2,\dots$

03$\quad$$\quad$$x'_{t}=P_{\mathcal{H}_{i,[t]}}(x'_{t-1}+\lambda_{t-K'}f_{i,[t]})$

04$\quad$$\quad$Let $\lambda_{t}$ be such that $x'_{t}+\lambda_{t}f_{i,[t]}=x'_{t-1}+\lambda_{t-K'}f_{i,[t]}$.

05$\quad$End for 
\end{lyxalgorithm}
If the $\lambda_{r}$ were chosen so that $\lambda_{r}=0$ for all
$r\in[1-K',0]$, then Algorithm \ref{alg:small-Dykstra} reduces to
Dykstra's algorithm. The $\{\lambda_{r}\}_{r=1-K'}^{0}$ are warmstarts
to Dykstra's algorithm, so we refer to Algorithm \ref{alg:small-Dykstra}
as a warmstart Dykstra's algorithm. The fact that $\lim_{t\to\infty}x'_{t}=P_{C_{i}}(x_{j}^{\circ}+e_{p(j),j})$
(even with the nonzero warmstart values $\lambda_{r}$ for $r\in[1-K,0]$)
follows from some simple changes to the Boyle-Dykstra theorem (see
\cite{Pang_DBAP}.)

The next result is a relationship between $x_{j}^{\circ}$ and $x_{j-1}^{+}$. 
\begin{prop}
\label{prop:Dykstra-leads-to-prj-form}(Dykstra steps leads to projection
form) Suppose that in Algorithm \ref{alg:classical-Dyk}, we have
$\|x_{0}^{+}\|\leq\min(\hat{\epsilon},\sqrt{2\bar{\epsilon}})$ and
that $\delta^{*}(y_{i}^{\circ},C_{i})=0$ for all $i\in[1,m]$. Recall
the definition of $K(\cdot,\cdot,\cdot)$ in Theorem \ref{thm:From-thm-3-3}.
We can find $E,D\in\{0,1\}^{\bar{t}}$ such that 
\[
x_{j}^{\circ}\in K(E,D,x_{j-1}^{+}),
\]
 
\begin{equation}
E_{t}=1\mbox{ for all }t\in[1,\bar{t}]\mbox{ and }D_{t}=\begin{cases}
1 & \mbox{ if }\lambda_{t}>0\\
0 & \mbox{ otherwise,}
\end{cases}\label{eq:E-t-and-D-t}
\end{equation}
and
\begin{enumerate}
\item For $t\in\{1,\dots,\bar{t}-1\}$, the $t$th hyperplane is $H_{i,[t]}$
(see \eqref{eq:starting-hyperplanes}), where $[t]$ is the integer
in $\{1,\dots,K'\}$ such that $K'$ divides $t-[t]$. 
\item The $\bar{t}$th hyperplane (i.e., the last hyperplane), which we
call $H_{j}$, is the intersection of some of the hyperplanes $H_{i,[t]}$
such that $E_{t}=D_{t}=1$. 
\end{enumerate}
Furthermore, the dual vector $e_{j,j}$ is in the conical hull of
$\{f_{i,[t]}:E_{t}=D_{t}=1\}$.\end{prop}
\begin{proof}
We make use of Proposition \ref{prop:warmstart-Dykstra-to-find-e-j}
and solve \eqref{eq:def-Dykstra-subpblm} using Algorithm \ref{alg:small-Dykstra}.
By Lemma \ref{lem:set-N-1}(2), $y''_{r'}=0$ for all $r'\in[K'+1,K]$,
so we can ignore the indices $K'+1$ to $K$. For example, the summations
``$\sum_{r'=1}^{K}$'' in \eqref{eq:defn-v-prime} can be replaced
by ``$\sum_{r'=1}^{K'}$'' instead. The starting variables $(y_{1}'',\dots,y_{K'}'')$
are chosen so that $y_{r}''=\lambda_{r-K'}f_{i,[r]}$ for $r\in[1,K']$,
$\{\lambda_{r}\}_{r=1-K'}^{0}$ are nonnegative numbers like in Algorithm
\ref{alg:small-Dykstra} and 
\begin{equation}
\begin{array}{c}
\underset{r'=1}{\overset{K'}{\sum}}y_{r'}''=e_{p(j),j},\end{array}\label{eq:starting-dual-Dykstra-subpblm}
\end{equation}
which would satisfy \eqref{eq:decompose-y-pp}. Let $\{x_{t}'\}_{t=0}^{\infty}$
be the iterates generated by Algorithm \ref{alg:small-Dykstra}. The
choice of $\{y_{r'}''\}_{r'=1}^{K'}$ gives $x_{0}'=x_{j-1}^{+}$.
By the convergence properties of a warmstart Dykstra's algorithm,
we have $\lim_{t\to\infty}x'_{t}=x_{j}^{\circ}$. 

Through Lemma \ref{lem:From-thm-2-4}, line 3 of Algorithm \ref{alg:small-Dykstra}
implies that $x_{t}'\in K(\{1\},\{0\},x_{t-1}')$, where the hyperplane
involved is $H_{i,[t]}$. Also, there are indices $t$ for which $\lambda_{t}>0$,
which implies that $x'_{t}=P_{H_{i,[t]}}(x'_{t-1}+\lambda_{t-K'}f_{i,[t]})$,
or $x_{t}'\in K(\{1\},\{1\},x_{t-1}')$. Define $T_{t}^{*}$ and $\bar{T}_{t}^{*}$
by 
\begin{eqnarray*}
T_{t}^{*} & := & \{t':t'\leq t,\lambda_{t'}>0\}\\
\bar{T}_{t}^{*} & := & \{[t']:t'\leq t,\lambda_{t'}>0\}.
\end{eqnarray*}
Since $\bar{T}_{t}^{*}\subset\{1,\dots,K'\}$ and is monotonically
increasing, there is some $\bar{t}$ such that $\bar{T}_{t}^{*}=\bar{T}_{\bar{t}-1}^{*}$
for all $t\geq\bar{t}-1$. Let $\bar{T}^{*}$ be $\bar{T}_{\bar{t}-1}^{*}$.
Using Theorem \ref{thm:From-thm-3-3}, we obtain 
\[
x_{\bar{t}-1}'\in K(E,D,x_{0}'),
\]
where $E,D\in\{0,1\}^{\bar{t}-1}$ are as defined in \eqref{eq:E-t-and-D-t}.
(We still have to resolve $E_{\bar{t}}$ and $D_{\bar{t}}$.) By Lemma
\ref{lem:compare-lem-3-4}, we can increase $\bar{t}$ if necessary
so that if $t\geq\bar{t}-1$, then 
\[
x_{j}^{\circ}-x'_{t}\in\sspan\{f_{i,[t']}:t'\in[t-K'+1,t],\lambda_{t'}>0\}.
\]
Moreover, $\bar{t}$ can be large enough so that if $t\geq\bar{t}-1$,
then $\lambda_{t'}>0$ and $t'\in[t-K'+1,t]$ implies $x_{j}^{\circ}\in H_{i,[t']}$.
This implies that the projection of $x_{\bar{t}-1}'$ onto $H_{j}$,
where $H_{j}$ is defined by 
\[
H_{j}=\cap\{H_{i,[t]}:t\in[\bar{t}-K',\bar{t}-1],\lambda_{t}>0\}.
\]
equals $x_{j}^{\circ}$. In other words $x_{j}^{\circ}=P_{H_{j}}(x_{\bar{t}-1}')$.
Since 
\[
\{[t]:t\in[\bar{t}-K',\bar{t}-1],\lambda_{t}>0\}\subset T^{*},
\]
 (2) holds, and the previous discussions show that (1) holds. 

Lastly, we show that $e_{j,j}=x_{j-1}^{+}+e_{p(j),j}-x_{j}^{\circ}$
is in $\mbox{cone}\{f_{i,[t']}:t'\in T^{*}\}$. Note that $x_{j-1}^{+}+e_{p(j),j}-x_{t}'$
equals $\sum_{r=t-K'+1}^{t}\lambda_{r}f_{i,[r]}$, which lies in $\mbox{cone}\{f_{i,[t']}:D_{t'}=1\}$
if $t$ is large enough. Since $\mbox{cone}\{f_{i,[t']}:t'\in T^{*}\}$
is a closed convex cone and $\lim_{t\to\infty}x_{t}'=x_{j}^{\circ}$
by Dykstra's algorithm, we have that $e_{j,j}=x_{j-1}^{+}+e_{p(j),j}-x_{j}^{\circ}$
is in $\mbox{cone}\{f_{i,[t]}:[t]\in\bar{T}^{*}\}$ as needed. 
\end{proof}
Next, we show that the similar thing happens to SHQP steps.
\begin{prop}
\label{prop:SHQP-leads-to-prj-form}(SHQP steps leads to projection
form) Consider the assumptions in Proposition \ref{prop:Dykstra-leads-to-prj-form}.
Suppose $\|x_{0}^{+}\|\leq\min(\hat{\epsilon},\sqrt{2\bar{\epsilon}})$
and that $\delta^{*}(y_{i}^{\circ},C_{i})=0$ for all $i\in[1,m]$.
Suppose that in the SHQP step of Algorithm \ref{alg:classical-Dyk},
whenever $i\in Q_{j}$, the polyhedron $P_{i,j}$ is the intersection
of halfspaces of the form 
\begin{itemize}
\item [(H1)]$H_{i,r}$, where $r\in[1,K']$ or 
\item [(H2)]$\{x:\langle e_{j',j'},x\rangle\geq0\}$ where $j'\leq j$
and $s(j')=i$. 
\end{itemize}
Then we have $x_{j}^{+}\in K(E,D,x_{j}^{\circ})$ for some $E,D\in\{0,1\}^{l}$,
where $E_{t}=1$ for all $t\in[1,l]$, and 
\begin{enumerate}
\item For $t\in\{1,\dots,l-1\}$, the $t$th hyperplane, which we call $\tilde{H}_{t}$,
is the boundary of a halfspace of either the kind (H1) or (H2) above. 
\item The $l$th hyperplane (i.e., last hyperplane) is the intersection
of some of the hyperplanes mentioned in the previous point (1) for
which $D_{t}=1$.
\end{enumerate}
\end{prop}
\begin{proof}
The proof is almost exactly the same as in Proposition \ref{prop:Dykstra-leads-to-prj-form}.
We show how to find the dual variables $(e_{\pi(j,1),j+1},\dots,e_{\pi(j,m),j+1})$
from $(e_{\pi(j,1),j},\dots,e_{\pi(j,m),j})$. Recall the definition
of $Q_{j}$. The SHQP step can be phrased as the problem of finding
$\{e_{\pi(j,i),j+1}\}_{i\in Q_{j}}$ from $\{e_{\pi(j,i),j}\}_{i\in Q_{j}}$. 

By repeating halfspaces defining each $P_{i,j}$, where $i\in Q_{j}$,
if necessary, we can assume that each $P_{i,j}$ is the intersection
of $\tilde{K}$ halfspaces. We label the halfspaces used to form each
$P_{i,j}$ by $\tilde{\mathcal{H}}_{i,r}$, where $r\in\{1,\dots,\tilde{K}\}$.
(Note that these halfspaces can be of the type (H2), and hence the
tilde.) Consider the optimization problem 
\begin{equation}
\underset{\{\tilde{y}_{i,r}\}_{i\in Q_{j},r\in[1,\tilde{K}]}}{\min}\begin{array}{c}
\frac{1}{2}\|d-\underset{i\notin Q_{j}}{\sum}e_{\pi(j,i),j}-\underset{i\in Q_{j}}{\sum}\underset{r=1}{\overset{\tilde{K}}{\sum}}\tilde{y}_{i,r}\|^{2}+\underset{i\in Q_{j}}{\sum}\underset{r=1}{\overset{\tilde{K}}{\sum}}\delta^{*}(\tilde{y}_{i,r},\mathcal{\tilde{H}}_{i,r})\end{array}\label{eq:SHQP-sub-QP}
\end{equation}
Let the starting $\{\tilde{y}_{i,r}\}_{i\in Q_{j},r\in[1,\tilde{K}]}$,
say $\{\tilde{y}_{i,r}^{\circ}\}_{i\in Q_{j},r\in[1,\tilde{K}]}$
be such that 
\[
\begin{array}{c}
\underset{r=1}{\overset{\tilde{K}}{\sum}}\tilde{y}_{i,r}^{\circ}=e_{\pi(j,i),j}\mbox{ and }\underset{r=1}{\overset{\tilde{K}}{\sum}}\delta^{*}(\tilde{y}_{i,r}^{\circ},\tilde{\mathcal{H}}_{i,r})=\delta^{*}(e_{\pi(j,i),j},C_{i}).\end{array}
\]
An optimal solution of \eqref{eq:SHQP-sub-QP}, say $\{\tilde{y}_{i,r}^{+}\}_{i\in Q_{j},r\in[1,\tilde{K}]}$,
would allow us to reconstruct $e_{\pi(j,i),j+1}$ by 
\[
e_{\pi(j,i),j+1}=\underset{r=1}{\overset{\tilde{K}}{\sum}}\tilde{y}_{i,r}^{+}.
\]
The primal iterates can be estimated using a warmstart Dykstra's algorithm
similar to that in the proof of Proposition \ref{prop:Dykstra-leads-to-prj-form}.
\end{proof}
Let $M$ and $N$ be closed subspaces in the Hilbert space $X$.
The angle between $M$ and $N$ is the angle between $0$ and $\pi/2$
whose cosine is given by 
\begin{eqnarray*}
c(M,N) & := & \sup\{|\langle x,y\rangle|:x\in M\cap[M\cap N]^{\perp},\|x\|\leq1\\
 &  & \phantom{\sup\{|\langle x,y\rangle|:}\,\,y\in N\cap[M\cap N]^{\perp},\|y\|\leq1\}.
\end{eqnarray*}
 This definition is due to Friedrichs \cite{Friedrichs}. 

We take the following two results concerning angles. 
\begin{thm}
\label{thm:Friedrich-angles}(Properties of $c(M,N)$) Let $M$ and
$N$ be closed subspaces of a Hilbert space $X$. We have the following
results regarding $c(M,N)$. 
\begin{enumerate}
\item \cite[Corollary 9.37]{Deustch01} If $M$ and $N$ are closed subspaces,
one of which has finite codimension, in the Hilbert space $X$, then
$c(M,N)<1$.
\item \cite[Lemma 9.5(8)]{Deustch01} $c(M,N)=0$ if $M\subset N$ or $N\subset M$.
\end{enumerate}
\end{thm}
We have the following result on the convergence rate of the method
of alternating projections when the sets involved are linear subspaces.
\begin{thm}
\label{thm:DH97}(Consequence of \cite[Theorem 2.7]{Deutsch_Hundal_alt_proj_subspace_rate_JMAA_97})
Let $M_{i}$ be linear subspaces for $i\in[1,k]$, $M=\cap_{1}^{k}M_{i}$,
and 
\begin{equation}
\alpha:=\bigg[1-\prod_{l=1}^{k-1}\underbrace{\left\{ 1-c^{2}(M_{l},\cap_{i=1}^{l-1}M_{i})\right\} }_{s_{l}^{2}}\bigg]^{1/2}.\label{eq:DH97-term}
\end{equation}
Then $\|P_{M_{k}}P_{M_{k-1}}\cdots P_{M_{1}}-P_{M}\|\leq\alpha$.\end{thm}
\begin{proof}
This is easily seen to be a particular case of \cite[Theorem 2.7]{Deutsch_Hundal_alt_proj_subspace_rate_JMAA_97}.
We refer to their result for the most general version.
\end{proof}
We have the following theorem, adapting the proof of \cite[Lemma 3.7]{Deutsch_Hundal_rate_Dykstra}.
\begin{thm}
\label{thm:DH94-lin-conv}(Asymptotic linear convergence) Recall the
conditions in Lemmas \ref{lem:set-N-1} and \ref{lem:compare-lem-3-4}.
Suppose in Algorithm \ref{alg:classical-Dyk}, we have $\|x_{0}^{+}\|\leq\min(\hat{\epsilon},\sqrt{2\bar{\epsilon}})$
and $\delta^{*}(y_{i}^{\circ},C_{i})=0$ for all $i$. Suppose further
that 
\begin{enumerate}
\item For all $i\in[1,m]$, the polyhedron $P_{i,j}$ is the intersection
of halfspaces of the form $\mathcal{H}_{i,r}$ in \eqref{eq:starting-halfspaces}
and halfspaces of the form $\{x:\langle e_{j',j'},x\rangle\leq0\}$,
where $j'<j$ is such that $s(j')=i$. 
\end{enumerate}
Then there is a constant $\rho\in[0,1)$ such that $\|x_{w'}^{+}\|\leq\rho\|x_{0}^{+}\|$. \end{thm}
\begin{proof}
Each $y_{i}^{+}$ can be written in terms of $\{y_{i,r}^{+}\}_{r=1}^{K'}$
so that 
\[
\begin{array}{c}
y_{i}^{+}=\underset{r=1}{\overset{K'}{\sum}}y_{i,r}^{+}\mbox{ and }\delta^{*}(y_{i}^{+},C_{i})=\underset{r=1}{\overset{K'}{\sum}}\delta^{*}(y_{i,r}^{+},\mathcal{H}_{i,r}).\end{array}
\]
Let 
\[
T=\{(i,r):y_{i,r}^{+}\neq0\}.
\]
From Lemma \ref{lem:compare-lem-3-4}, we have $x_{w'}^{+}\in\sspan\{f_{i,r}:(i,r)\in T\}$.

Next, we make use of Theorem \ref{thm:From-thm-3-3} and Propositions
\ref{prop:Dykstra-leads-to-prj-form} and \ref{prop:SHQP-leads-to-prj-form}
to see that 
\[
x_{w'}^{+}\in K(E,D,x_{0}^{+})
\]
for some $E,D\in\{0,1\}^{l}$. Let the hyperplanes defined in Propositions
\ref{prop:Dykstra-leads-to-prj-form} and \ref{prop:SHQP-leads-to-prj-form}
for each $t\in\{1,\dots,l\}$ be $\tilde{H}_{t}$. Moreover, if $\tilde{H}_{t}$
equals $H_{i,r}$ for some $i\in[1,m]$ and $r\in[1,K']$, we refer
to the normal vector $f_{i,r}$ as $\tilde{f}_{t}$. Since $x_{w'}^{+}\in\sspan\{f_{i,r}:(i,r)\in T\}$,
we recall the definition of $K(\cdot,\cdot,\cdot)$ to get 
\begin{eqnarray}
x_{w'}^{+} & \in & P_{\scriptsize\sspan\{f_{i,r}:(i,r)\in T\}}K(E,D,x_{0}^{+})\nonumber \\
 & \overset{\eqref{eq:the-K}}{=} & P_{\scriptsize\sspan\{f_{i,r}:(i,r)\in T\}}\co\{\bar{P}_{S}(x_{0}^{+}):E\geq S\geq D\}\nonumber \\
 & \overset{\scriptsize\mbox{Lem. }\ref{lem:co-P-commute}}{=} & \co\{P_{\scriptsize\sspan\{f_{i,r}:(i,r)\in T\}}\bar{P}_{S}(x_{0}^{+}):E\geq S\geq D\}.\nonumber \\
\Rightarrow\|x{}_{w'}^{+}\| & \leq & \|x_{0}^{+}\|\max\{\|P_{\scriptsize\sspan\{f_{i,r}:(i,r)\in T\}}\bar{P}_{S}\|:E\geq S\geq D\}\label{eq:norm-x-w}
\end{eqnarray}
One can retrace from conditions (1) and (2) of both Propositions \ref{prop:Dykstra-leads-to-prj-form}
and \ref{prop:SHQP-leads-to-prj-form} that $\sspan\{f_{i,r}:(i,r)\in T\}\subset\sspan\{\tilde{H}_{t}^{\perp}:D_{t}=1\}$.
Also, if $S\in\{0,1\}^{l}$ is such that $S\geq D$, then $\sspan\{\tilde{H}_{t}^{\perp}:D_{t}=1\}\subset\sspan\{\tilde{H}_{t}^{\perp}:S_{t}=1\}$.
Combining with \eqref{eq:norm-x-w} gives 

\[
\|x{}_{w'}^{+}\|\leq\|x_{0}^{+}\|\max\{\|P_{\scriptsize\sspan\{\tilde{H}_{t}^{\perp}:S_{t}=1\}}\bar{P}_{S}\|:E\geq S\geq D\}.
\]
Now, since $\sspan\{\tilde{H}_{t}^{\perp}:S_{t}=1\}=[\cap\{\tilde{H}_{t}:S_{t}=1\}]^{\perp}$,
we have 
\begin{equation}
P_{\scriptsize\sspan\{\tilde{H}_{t}^{\perp}:S_{t}=1\}}=I-P_{\cap\{\tilde{H}_{t}:S_{t}=1\}}.\label{eq:subsp-break-1}
\end{equation}
Therefore, 
\begin{equation}
P_{\scriptsize\sspan\{\tilde{H}_{t}^{\perp}:S_{t}=1\}}\bar{P}_{S}\overset{\eqref{eq:subsp-break-1}}{=}[I-P_{\cap\{\tilde{H}_{t}:S_{t}=1\}}]\bar{P}_{S}\overset{\eqref{eq:def-bar-P}}{=}\bar{P}_{S}-P_{\cap\{\tilde{H}_{t}:S_{t}=1\}}.\label{eq:subsp-break-2}
\end{equation}
Reordering the sequence $\{t\in\{1,\dots,l\}:S_{t}=1\}$ as $\{t_{1},t_{2},\dots,t_{\bar{l}(S)}\}$
gives 
\begin{eqnarray}
\|P_{\scriptsize\sspan\{\tilde{H}_{t}^{\perp}:S_{t}=1\}}\bar{P}_{S}\| & \overset{\eqref{eq:subsp-break-2}}{=} & \|\bar{P}_{S}-P_{\cap\{\tilde{H}_{t}:S_{t}=1\}}\|\label{eq:apply-DH97-form}\\
 & = & \|P_{\tilde{H}_{t_{\bar{l}(S)}}}P_{\tilde{H}_{t_{\bar{l}(S)-1}}}\cdots P_{\tilde{H}_{t_{2}}}P_{\tilde{H}_{t_{1}}}-P_{\cap\{\tilde{H}_{t_{l}}:1\leq l\leq\bar{l}(S)\}}\|.\nonumber 
\end{eqnarray}
We now apply Theorem \ref{thm:DH97} to estimate the last term in
\eqref{eq:apply-DH97-form}. We look at what the $\alpha$ in \eqref{eq:DH97-term}
would be for \eqref{eq:apply-DH97-form}. For this sequence $\{t_{1},\dots,t_{\bar{l}(S)}\}$,
let $Q$ be the subset defined by 
\[
Q:=\{l:\tilde{H}_{t_{l}}=H_{i,r}\mbox{ for some }i\in[1,m]\backslash I\mbox{ and }r\in[1,K']\}.
\]
\textbf{Claim:} If $l\notin Q$, then 
\begin{equation}
\tilde{H}_{t_{l}}\supset\cap_{l'=1}^{l-1}\tilde{H}_{t_{l'}}.\label{eq:hyperplane-bigger-than-intersection}
\end{equation}
If $l\notin Q$, then $\tilde{H}_{t_{l}}$ is a hyperplane of the
type in Proposition \ref{prop:Dykstra-leads-to-prj-form}(2), or of
the type in Proposition \ref{prop:SHQP-leads-to-prj-form}(1)(2).
If $\tilde{H}_{t_{l}}$ is a hyperplane of the type in Proposition
\ref{prop:Dykstra-leads-to-prj-form}(2), then \eqref{eq:hyperplane-bigger-than-intersection}
holds since the last hyperplane is the intersection of hyperplanes
for which $D_{t}=1$.  Condition \eqref{eq:hyperplane-bigger-than-intersection}
holds for the hyperplanes of the type in Proposition \ref{prop:SHQP-leads-to-prj-form}(1)
as well, this time making use of the fact that $e_{j,j}$ is in the
conical hull of $\{f_{i,[t]}:E_{t}=D_{t}=1\}$ at the end of Proposition
\ref{prop:Dykstra-leads-to-prj-form}. Lastly, Condition \eqref{eq:hyperplane-bigger-than-intersection}
holds for the hyperplanes of the type in Proposition \ref{prop:SHQP-leads-to-prj-form}(2)
for the same reason that it holds for the hyperplane of the type in
Proposition \ref{prop:Dykstra-leads-to-prj-form}(2). The proof of
the claim is complete.

If $l\notin Q$, then 
\begin{equation}
s_{l}^{2}\overset{\eqref{eq:DH97-term}}{=}1-c^{2}(H_{t_{l}},\cap_{i=1}^{l-1}H_{t_{i}})\overset{\eqref{eq:hyperplane-bigger-than-intersection},\scriptsize\mbox{Thm }\ref{thm:Friedrich-angles}(2)}{=}1.\label{eq:s2-is-one}
\end{equation}
If $l\in Q$, then 
\begin{equation}
s_{l}^{2}\overset{\eqref{eq:DH97-term}}{=}1-c^{2}(H_{t_{l}},\cap_{i=1}^{l-1}H_{t_{i}})\overset{\eqref{eq:hyperplane-bigger-than-intersection}}{=}1-c^{2}(H_{t_{l}},\cap_{{1\leq i\leq l-1\atop i\in Q}}H_{t_{i}}).\label{eq:s-l-finitely-many-forms}
\end{equation}
The formulas and \eqref{eq:s2-is-one} and \eqref{eq:s-l-finitely-many-forms}
for when $l\notin Q$ and $l\in Q$, together with Theorem \eqref{thm:Friedrich-angles}(2),
implies that only finitely many of the $s_{l}^{2}$ are less than
$1$, and that each $s_{l}^{2}$ takes only finitely many possibilities.
Thus the term $\alpha$ in \eqref{eq:DH97-term} takes on only finitely
many possibilities in $[0,1)$, so there is a constant $\rho\in[0,1)$
such that the last term in \eqref{eq:apply-DH97-form} lies in $[0,\rho]$.
This ends our proof.
\end{proof}

\section{Nonasymptotic convergence properties in polyhedral problems}

In this section, we recall Assumption \ref{assm:poly-C-i} and look
at the nonasymptotic convergence properties in polyhedral problems.

We prove a lower bound on the decrease in the dual objective function
in one cycle.
\begin{prop}
(Estimate of decrease in dual objective function) Recall Assumption
\ref{assm:poly-C-i} and Algorithm \ref{alg:classical-Dyk}. Recall
also that $x_{0}^{+}=d-\sum_{i=1}^{m}y_{i}^{\circ}$. Suppose $y_{i}^{\circ}$
can be written in terms of $\tilde{y}_{i,r}$ (where $i\in[1,m]$,
$r\in[1,K]$) so that \begin{subequations}\label{eq_m:break-up-normal-of-polyhedron}
\begin{eqnarray}
y_{i}^{\circ} & = & \!\!\!\!\begin{array}{c}
\underset{r=1}{\overset{K}{\sum}}\tilde{y}_{i,r},\end{array}\label{eq:break-up-normal-of-polyhedron-1}\\
\mbox{ and }\delta^{*}(y_{i}^{\circ},C_{i}) & = & \begin{array}{c}
\!\!\!\!\underset{r=1}{\overset{K}{\sum}}\delta^{*}(\tilde{y}_{i,r},\mathcal{H}_{i,r}).\end{array}\label{eq:break-up-normal-of-polyhedron-2}
\end{eqnarray}
\end{subequations}Then 
\begin{equation}
\begin{array}{c}
v(y^{\circ})-v(y^{+})\geq\frac{1}{2w'-1}\underset{{i\in\{1,\dots,m\}\atop r\in\{1,\dots,K\}}}{\max}\big\{ d(x_{0}^{+},\mathcal{H}_{i,r}),\min\{d(x_{0}^{+},H_{i,r}),\|\tilde{y}_{i,r}\|\}\big\}^{2}.\end{array}\label{eq:dist-est-formula}
\end{equation}
\end{prop}
\begin{proof}
For each $i\in\{1,\dots,m\}$ and $r\in\{1,\dots,L\}$, we seek to
show that $v(y^{\circ})-v(y^{+})\geq\frac{1}{2w'-1}d(x_{0}^{+},\mathcal{H}_{i,r})^{2}$.
Define $n(i)$ to be $n(i)=\min\{j'\geq0:s(j')=i\}$. (I.e., $n(i)$
is the first positive index $j'$ such that $s(j')=i$.) Note that
$n(i)\leq w'$. We thus have 
\begin{eqnarray}
v(y^{+})-v(y^{\circ}) & \overset{\scriptsize\mbox{Cor }\ref{cor:decrease-after-w-steps}}{\leq} & \!\!\!\!\begin{array}{c}
-\underset{j=1}{\overset{w'}{\sum}}[\|x_{j}^{\circ}-x_{j}^{+}\|^{2}+\|x_{j-1}^{+}-x_{j}\|^{2}]\end{array}\label{eq:dual-fn-dec-chain}\\
 & \leq & \!\!\!\!\begin{array}{c}
-\|x_{0}^{+}-x_{1}^{\circ}\|^{2}-\underset{j=1}{\overset{n(i)-1}{\sum}}[\|x_{j}^{\circ}-x_{j}^{+}\|^{2}+\|x_{j}^{+}-x_{j+1}^{\circ}\|^{2}]\end{array}\nonumber \\
 & \leq & \!\!\!\!\begin{array}{c}
-\frac{1}{2n(i)-1}\left[\|x_{0}^{+}-x_{1}^{\circ}\|+\underset{j=1}{\overset{n(i)-1}{\sum}}[\|x_{j}^{\circ}-x_{j}^{+}\|+\|x_{j}^{+}-x_{j+1}^{\circ}\|]\right]^{2}\end{array}\nonumber \\
 & \leq & \!\!\!\!\begin{array}{c}
-\frac{1}{2n(i)-1}\|x_{0}^{+}-x_{n(i)}^{\circ}\|^{2}.\end{array}\nonumber 
\end{eqnarray}
For any $r\in[1,K']$, note that the primal iterate $x_{n(i)}^{\circ}$
lies in $C_{i}$, and hence $\mathcal{H}_{i,r}$, so $\|x_{0}^{+}-x_{n(i)}^{\circ}\|\geq d(x_{0}^{+},\mathcal{H}_{i,r})$.
Together with \eqref{eq:dual-fn-dec-chain}, we have $v(y^{\circ})-v(y^{+})\geq\frac{1}{2w'-1}d(x_{0}^{+},\mathcal{H}_{i,r})^{2}$
for all $i$, which addresses the first term in the maximum in \eqref{eq:dist-est-formula}.

Next, we show that for each $i\in\{1,\dots,m\}$ and $r\in\{1,\dots,L\}$,
\begin{equation}
v(y^{\circ})-v(y^{+})\geq\frac{1}{2w'-1}\min\{d(x_{0}^{+},H_{i,r}),\|\tilde{y}_{i,r}\|\}^{2},\label{eq:2nd-part-to-prove}
\end{equation}
which would complete the proof of this result. Fix some $(i,r)$
such that $\|\tilde{y}_{i,r}\|>0$. When $x_{0}^{+}\notin\mathcal{H}_{i,r}$,
we recall that $x_{n(i)}^{\circ}\in C_{i}\subset\mathcal{H}_{i,r}$.
This would imply that $\|x_{0}^{+}-x_{n(i)}^{\circ}\|\geq d(x_{0}^{+},H_{i,r})$,
which gives 
\begin{equation}
\underset{j=1}{\overset{w'}{\sum}}[\|x_{j}^{\circ}-x_{j}^{+}\|^{2}+\|x_{j-1}^{+}-x_{j}\|]\geq d(x_{0}^{+},H_{i,r}).\label{eq:trig-ineq-all-terms}
\end{equation}
Another case when \eqref{eq:trig-ineq-all-terms} holds is when $x_{0}^{+}\in\mathcal{H}_{i,r}$
and there is some $j^{*}\in\{1,\dots,w'\}$ such that $x_{j^{*}}^{\circ}\notin\mathcal{H}_{i,r}$
or $x_{j^{*}}^{+}\notin\mathcal{H}_{i,r}$. If \eqref{eq:trig-ineq-all-terms}
holds, an argument similar to \eqref{eq:dual-fn-dec-chain} gives
\begin{equation}
\begin{array}{c}
v(y^{\circ})-v(y^{+})\geq\frac{1}{2w'-1}d(x_{0}^{+},H_{i,r})^{2}\mbox{ for all }i,\end{array}\label{eq:dist-address-2nd-term-in-max}
\end{equation}
which implies \eqref{eq:2nd-part-to-prove}. 

It remains to prove \eqref{eq:2nd-part-to-prove} for the case when
both $x_{j}^{\circ}$ and $x_{j}^{+}$ lie in $\mathcal{H}_{i,r}$
for all $j\in\{1,\dots,w'\}$.  Recall that the term $x_{n(i)}^{\circ}$
is found from $e_{n(i),n(i)}$, where $e_{n(i),n(i)}$ is obtained
by the method in Proposition \ref{prop:warmstart-Dykstra-to-find-e-j}.

The problem \eqref{eq:defn-v-prime} can be solved by a warmstart
Dykstra's algorithm (i.e., a block coordinate minimization of the
coordinates) where one warmstarts with $(y''_{1},\dots,y''_{K})=(\tilde{y}_{i,1},\dots\tilde{y}_{i,K})$.
Suppose we now minimize the $r$th coordinate to get 
\[
y_{r}^{+}:=\underset{y}{\arg\min}\,\,v'(y''_{1},\dots y''_{r-1},y,y''_{r+1},\dots,y''_{K}).
\]
For convenience, let $j=n(i)$. We label the resulting primal variable
as $x_{n(i)}'$, which can be written in two ways 
\begin{eqnarray}
x_{n(i)}' & = & d-\underset{{1\leq r'\leq K\atop r'\neq r}}{\sum}\tilde{y}_{i,r'}-y_{i,r}^{+}-\underset{{1\leq i'\leq m\atop i'\neq i}}{\sum}e_{\pi(j,i'),j}\nonumber \\
 & = & P_{\mathcal{H}_{i,r}}\big(\underbrace{d-\underset{{1\leq r'\leq K\atop r'\neq r}}{\sum}\tilde{y}_{i,r'}-\underset{{1\leq i'\leq m\atop i'\neq i}}{\sum}e_{\pi(j,i'),j}}_{d'}\big).\label{eq:d-prime-appears}
\end{eqnarray}
Define $(y_{1}',\dots,y_{K}')$ to be a minimizer of $v'(\cdot)$
like in \eqref{eq:def-Dykstra-subpblm}. We have 
\begin{equation}
v'(y_{1}',\dots,y_{K}')\leq v'(y''_{1},\dots y''_{r-1},y_{r}^{+},y''_{r+1},\dots,y''_{K}).\label{eq:dual-dec-many-to-one}
\end{equation}
From the definitions of $v(\cdot)$ and $v'(\cdot)$ in \eqref{eq:D-prime-d-def}
and \eqref{eq:defn-v-prime}, we have\begin{subequations} 
\begin{eqnarray}
 &  & \begin{array}{c}
v'(y_{1}'',\dots,y_{K}'')+\underset{{1\leq i'\leq m\atop i'\neq i}}{\sum}\delta^{*}(e_{\pi(j-1,i'),j},C_{i'})=v(e_{\pi(j-1,1),j},\dots e_{\pi(j-1,m),j}),\qquad\end{array}\label{eq:v-prime-1}\\
 &  & \begin{array}{c}
v'(y_{1}',\dots,y_{K}')+\underset{{1\leq i'\leq m\atop i'\neq i}}{\sum}\delta^{*}(e_{\pi(j,i'),j},C_{i'})=v(e_{\pi(j,1),j},\dots e_{\pi(j,m),j}).\end{array}\label{eq:v-prime-2}
\end{eqnarray}
\end{subequations}Note that since $i'\neq i=s(j)$, we have $e_{\pi(j,i'),j}=e_{\pi(j-1,i'),j}$,
which implies 
\begin{equation}
\begin{array}{c}
\underset{{1\leq i'\leq m\atop i'\neq i}}{\sum}\delta^{*}(e_{\pi(j-1,i'),j},C_{i'})=\underset{{1\leq i'\leq m\atop i'\neq i}}{\sum}\delta^{*}(e_{\pi(j,i'),j},C_{i'}).\end{array}\label{eq:the-sums-equal}
\end{equation}
 By using the methods in Proposition \ref{prop:Decrease-in-dual},
we have 
\begin{equation}
v'(y''_{1},\dots y''_{r-1},y_{r}^{+},y''_{r+1},\dots,y''_{K})-v'(y''_{1},\dots,y''_{K})\leq-\|x_{n(i)}'-x_{n(i)-1}^{+}\|^{2}.\label{eq:dec-primal-one-coord}
\end{equation}
These give the following chain of inequalities 
\begin{eqnarray*}
 &  & v(y^{+})-v(y^{\circ})\\
 & \leq & v(e_{\pi(j,1),j},\dots e_{\pi(j,m),j})-v(y^{\circ})\\
 & \overset{\eqref{eq:v-prime-2}}{=} & v'(y_{1}',\dots,y_{K}')+\underset{{1\leq i'\leq m\atop i'\neq i}}{\sum}\delta^{*}(e_{\pi(j,i'),j},C_{i'})-v(y^{\circ})\\
 & \overset{\eqref{eq:dual-dec-many-to-one}}{\leq} & v'(y''_{1},\dots y''_{r-1},y_{r}^{+},y''_{r+1},\dots,y''_{K})+\underset{{1\leq i'\leq m\atop i'\neq i}}{\sum}\delta^{*}(e_{\pi(j,i'),j},C_{i'})-v(y^{\circ})\\
 & \overset{\eqref{eq:v-prime-1},\eqref{eq:the-sums-equal}}{=} & v'(y''_{1},\dots y''_{r-1},y_{r}^{+},y''_{r+1},\dots,y''_{K})-v'(y_{1}'',\dots,y_{K}'')+v(e_{\pi(j-1,1),j},\dots e_{\pi(j-1,m),j})-v(y^{\circ})\\
 & \overset{\eqref{eq:dec-primal-one-coord},\eqref{eq_m:dual-fn-decrease}}{\leq} & -\|x_{n(i)-1}^{+}-x_{n(i)}'\|^{2}-\sum_{j=1}^{n(i)-1}[\|x_{j-1}^{+}-x_{j}^{\circ}\|^{2}+\|x_{j}^{\circ}-x_{j}^{+}\|^{2}]\\
 & \leq & -\frac{1}{2n(i)-1}\Bigg[\underbrace{\|x_{n(i)-1}^{+}-x_{n(i)}'\|+\sum_{j=1}^{n(i)-1}[\|x_{j-1}^{+}-x_{j}^{\circ}\|^{2}+\|x_{j}^{\circ}-x_{j}^{+}\|]}_{\gamma}\Bigg]^{2}.
\end{eqnarray*}
To simplify discussions, let $d'$ be the point marked in \eqref{eq:d-prime-appears}.
The point $d'$ also equals $x_{n(i)-1}'+\tilde{y}_{i,r}$. If $d'\in\mathcal{H}_{i,r}$,
then $x'_{n(i)}=d'$ and $x_{n(i)-1}^{+}-x'_{n(i)}=r$, in which case
$v(y^{+})-v(y^{\circ})\leq-\|\tilde{y}_{i,r}\|^{2}\leq-\frac{1}{2w'-1}\|\tilde{y}_{i,r}\|^{2}$.
But if $d'\notin\mathcal{H}_{i,r}$, then $x_{n(i)}'\in H_{i,r}$,
in which case the term $\gamma$ marked above satisfies $\gamma\geq d(x_{0}^{+},H_{i,r})$,
and the argument in \eqref{eq:dual-fn-dec-chain} can be repeated
to prove $v(y^{+})-v(y^{\circ})\leq-\frac{1}{2w'-1}d(x_{0}^{+},H_{i,r})^{2}$.
This ends our proof.
\end{proof}
Next, we prove the following.
\begin{prop}
\label{prop:deltas-in-dual-decrease}(Decrease in dual function when
sufficiently far from $0$) Recall $x_{0}^{+}=d-\sum_{i=1}^{m}y_{i}^{\circ}$,
$\tilde{y}_{i,r}$ are defined as in \eqref{eq_m:break-up-normal-of-polyhedron},
and $P_{C}(d)=0$. For any $\delta_{1}>0$, we can find $\delta_{2}>0$
such that if $\|x_{0}^{+}\|>\delta_{1}$, then the formula \eqref{eq:dist-est-formula}
is bounded from below by a constant $\delta_{2}>0$.\end{prop}
\begin{proof}
Recall 
\[
\mathcal{H}_{i,r}:=\{\tilde{x}:\langle f_{i,r},\tilde{x}\rangle\leq c_{i,r}\}.
\]

Let $\delta_{1}>0$ and suppose $x_{0}^{+}$ is such that $\|x_{0}^{+}\|>\delta_{1}$.
We recall the following fact that can be inferred from Lemma \ref{lem:sens-anal-QP}:
\begin{enumerate}
\item For all $I'\subset\{1,\dots,m\}\times\{1,\dots,K\}$, let $\bar{x}_{I'}$
be defined by $P_{\cap_{(i,r)\in I'}\mathcal{H}_{i,r}}(d)$. For any
choice of $I'\subset\{1,\dots,m\}\times\{1,\dots,K\}$ such that $\{f_{i,r}:(i,r)\in I'\}$
is linearly independent, $e\in\mathbb{R}^{n}$ and $\tilde{c}_{i,r}\in\mathbb{R}$
for all $(i,r)\in I'$, let $x$ be defined to be the projection of
$d-e$ onto 
\begin{equation}
\cap_{(i,r)\in I'}\{\tilde{x}:\langle f_{i,r},\tilde{x}\rangle\leq\tilde{c}_{i,r}\}.\label{eq:new-polyhedron}
\end{equation}
Then for any $\delta_{3}>0$, there exists $\delta_{2}>0$ such that
\begin{equation}
\|e\|\leq mK\delta_{2}\mbox{ and }|\tilde{c}_{i,r}-c_{i,r}|\leq\delta_{2}\mbox{ for all }(i,r)\in I'\mbox{ implies }\|x-\bar{x}_{I'}\|\leq\delta_{3}.\label{eq:condn-for-delta-2}
\end{equation}
 
\end{enumerate}
Let 
\[
\bar{x}_{I'}=P_{\cap_{(i,r)\in I'}\mathcal{H}_{i,r}}(d).
\]
If $\bar{x}_{I'}\neq0$, then $\bar{x}_{I'}\notin C$. Then for these
$I'$, there are halfspaces $\mathcal{H}_{i,r}$, where $i\in\{1,\dots,m\}$
and $r\in\{1,\dots,K\}$, for which $d(\bar{x}_{I'},\mathcal{H}_{i,r})>0$.
Let 
\[
\delta_{4}:=\min\big\{ d(\bar{x}_{I'},\mathcal{H}_{i,r}):d(\bar{x}_{I'},\mathcal{H}_{i,r})>0,i\in\{1,\dots m\},r\in\{1,\dots,K\}\big\}.
\]
Let $\delta_{3}=\min\{\delta_{4}/2,\delta_{1}\}$, and let $\delta_{2}>0$
be chosen such that \eqref{eq:condn-for-delta-2} holds.

Let the formula in the right hand side of \eqref{eq:dist-est-formula}
be $F$. We now prove that if $\|x_{0}^{+}\|>\delta_{1}$, then $F>\delta_{2}$.
Seeking a contradiction, suppose $F\leq\delta_{2}$. For the decomposition
of $y_{i}^{\circ}$ satisfying \eqref{eq_m:break-up-normal-of-polyhedron},
let $I^{\circ}:=\{(i,r):\|\tilde{y}_{i,r}\|>\delta_{2}\}$. Since
\eqref{eq:dist-est-formula} is satisfied, we must have 
\[
d(x_{0}^{+},H_{i,r})\leq\delta_{2}\mbox{ for all }(i,r)\in I^{\circ}.
\]
Recall $x_{0}^{+}=d-\sum_{i=1}^{m}y_{i}^{\circ}$, and that the $y_{i}^{\circ}$
can be decomposed as $y_{i}^{\circ}=\sum_{r=1}^{K}\tilde{y}_{i,r}$
satisfying \eqref{eq_m:break-up-normal-of-polyhedron}. We then write
\begin{equation}
\begin{array}{c}
x_{0}^{+}=d-\underset{i=1}{\overset{m}{\sum}}\underset{r=1}{\overset{K}{\sum}}\tilde{y}_{i,r}=d-\underset{(i,r)\in I^{\circ}}{\sum}\tilde{y}_{i,r}-\underset{(i,r)\notin I^{\circ}}{\sum}\tilde{y}_{i,r}.\end{array}\label{eq:break-to-2-types-1}
\end{equation}
By Caratheodory's theorem, we can find a subset $I'\subset I^{\circ}$
and $\alpha_{i}'\geq0$ for all $i\in I'$ such that $y_{i,r}'=\alpha_{i,r}'f_{i,r}$
for all $(i,r)\in I'$, $\{f_{i,r}:(i,r)\in I'\}$ is linearly independent,
and 
\begin{equation}
\begin{array}{c}
\underset{(i,r)\in I^{\circ}}{\sum}\tilde{y}_{i,r}=\underset{(i,r)\in I'}{\sum}y_{i,r}'.\end{array}\label{eq:break-to-2-types-2}
\end{equation}
Hence 
\begin{equation}
\begin{array}{c}
x_{0}^{+}\overset{\eqref{eq:break-to-2-types-1}}{=}d-\underset{(i,r)\in I^{\circ}}{\sum}\tilde{y}_{i,r}-\underset{(i,r)\notin I^{\circ}}{\sum}\tilde{y}_{i,r}\overset{\eqref{eq:break-to-2-types-2}}{=}d-\underset{(i,r)\in I'}{\sum}y'_{i,r}-\underset{(i,r)\notin I^{\circ}}{\sum}\tilde{y}_{i,r}.\end{array}\label{eq:KKT-of-proj}
\end{equation}
Then $x_{0}^{+}$ is the projection of $d-\sum_{(i,r)\notin I^{\circ}}\tilde{y}_{i,r}$
onto halfspaces $\cap_{(i,r)\in I'}\{\tilde{x}:\langle f_{i,r},\tilde{x}\rangle\leq\langle f_{i,r},x_{0}^{+}\rangle\}$.
(To see this, note that the nonzero multipliers $y_{i,r}'$ correspond
to halfspaces tight at $x_{0}^{+}$ and that \eqref{eq:KKT-of-proj}
is satisfied.) Moreover, recall that the halfspaces $\mathcal{H}_{i,r}$
are of the form 
\[
\mathcal{H}_{i,r}:=\{\tilde{x}:\langle f_{i,r},\tilde{x}\rangle\leq c_{i,r}\}.
\]
Since $d(x_{0}^{+},H_{i,r})\leq\delta_{2}$ for all $(i,r)\in I^{\circ}$,
we have $|\langle f_{i,r},x_{0}^{+}\rangle-c_{i,r}|\leq\delta_{2}$
for all $(i,r)\in I^{\circ}$. ($\langle f_{i,r},x_{0}^{+}\rangle$
plays the role of $\tilde{c}_{i,r}$.) Note that 
\[
\begin{array}{c}
\Big\| d-\Big(d-\underset{(i,r)\notin I^{\circ}}{\sum}\tilde{y}_{i,r}\Big)\Big\|=\Big\|\underset{(i,r)\notin I^{\circ}}{\sum}\tilde{y}_{i,r}\Big\|\leq mK\delta_{2}.\end{array}
\]
By the choice of $\delta_{2}>0$ that satisfies property (1) above
and the definition of $\bar{x}_{I'}$, we have $\|x_{0}^{+}-\bar{x}_{I'}\|\leq\min\{\delta_{4}/2,\delta_{1}\}$.
If $\bar{x}_{I'}\neq0$, then there is a halfspace $\mathcal{H}_{i,r}$
such that $d(\bar{x}_{I'},\mathcal{H}_{i,r})\geq\delta_{4}$, in which
case 
\[
d(x_{0}^{+},\mathcal{H}_{i,r})\geq d(\bar{x}_{I'},\mathcal{H}_{i,r})-\|x_{0}^{+}-\bar{x}_{I'}\|\geq\delta_{4}/2.
\]
In the case where $\bar{x}_{I'}=0$, then $\|x_{0}^{+}\|\leq\delta_{1}$,
which contradicts the choice of $\|x_{0}^{+}\|>\delta_{1}$. Thus
we are done. 
\end{proof}
We now prove an elementary lemma involving projections onto polyhedra. 
\begin{lem}
\label{lem:sens-anal-QP}(Sensitivity analysis of projections onto
polyhedra) Let $A\in\mathbb{R}^{m\times n}$ and $b\in\mathbb{R}^{m}$,
and assume that $A$ has linearly independent rows. Define the set
$S$ by $\{x:Ax\leq b\}$. For $\tilde{b}\in\mathbb{R}^{m}$, define
$\tilde{S}$ by $\{x:Ax\leq\tilde{b}\}$. Let $d$ and $\tilde{d}$
be in $\mathbb{R}^{n}$. For any $\delta_{3}>0$, there exists $\delta_{2}>0$
such that if $\|b-\tilde{b}\|_{\infty}\leq\delta_{2}$ and $\|d-\tilde{d}\|_{2}\leq\delta_{2}$,
then $\|P_{S}(d)-P_{\tilde{S}}(\tilde{d})\|\leq\delta_{3}$. \end{lem}
\begin{proof}
The well known result on the nonexpansiveness of the projections gives
us $\|P_{\tilde{S}}(d)-P_{\tilde{S}}(\tilde{d})\|\leq\|d-\tilde{d}\|\leq\delta_{2}$.
Suppose $\delta_{2}\leq\delta_{3}/2$. 

Next, we prove $\|P_{S}(d)-P_{\tilde{S}}(d)\|\leq\delta_{3}/2$. Let
$e\in\mathbb{R}^{m}$ be the vector of all ones. The smallest feasible
region under the condition $\|\tilde{b}-b\|_{\infty}\leq\delta_{2}$
is attained when $\tilde{b}=b-\delta_{2}e$. This gives an upper bound
of $\|d-P_{\tilde{S}}(d)\|$, which we call $U$. Similarly, when
$\tilde{b}=b+\delta_{2}e$, then we get the lower bound of $\|d-P_{\tilde{S}}(d)\|$,
which we call $L$. Let the $P_{\tilde{S}}(d)$ obtained in this case
be $p^{*}$. For all other possible cases, $\|d-P_{\tilde{S}}(d)\|\in[L,U]$.
Therefore both $P_{S}(d)$ and $P_{\tilde{S}}(d)$ lie in the sphere
with center $d$ and radius $U$, and in the halfspace $\{x:\langle d-p^{*},x-p^{*}\rangle\leq0\}$.
One can use trigonometry to calculate that this region has diameter
$2\sqrt{L^{2}-U^{2}}$. This quantity goes to $0$ as $\delta_{2}\searrow0$,
so we can make $\delta_{2}$ small enough so that $\|P_{\tilde{S}}(d)-P_{S}(d)\|\leq\delta_{3}/2$.
Combining the previous paragraph completes the proof of this result. 
\end{proof}
The next result shows the nonasymptotic convergence rate of the main
algorithm, Algorithm \ref{alg:main-alg}.
\begin{prop}
(Transitioning to asymptotic linear convergence) Consider Algorithm
\ref{alg:main-alg} being run on an instance of \eqref{eq:P-primal}.
Suppose Algorithm \ref{alg:classical-Dyk} is run so that 
\begin{enumerate}
\item For all $i\in[1,m]$, if $i\in Q_{j}$, then the polyhedra $P_{i,j}$
are chosen to be the intersection of halfspaces that were produced
by the projection process so far. 
\item The $w'$ over all calls in Algorithm \ref{alg:classical-Dyk} are
uniformly bounded by some $w$. 
\end{enumerate}
Then for an instance of the BAP \eqref{eq:P-primal} and a starting
$y^{0}$ in Algorithm \ref{alg:main-alg}, there are $\delta_{2}>0$,
$\bar{k}>0$ and $\rho\in[0,1)$ such that 
\begin{enumerate}
\item [(A)] If $k<\bar{k}$, then $v(y^{k})<v(y^{k-1})-\delta_{2}$, and 
\item [(B)]If $k>\bar{k}$, then $v(y^{k})\leq\rho v(y^{k-1})$. 
\end{enumerate}
\end{prop}
\begin{proof}
Let $\bar{\epsilon}$ and $\hat{\epsilon}$ be as defined in Lemmas
\ref{lem:set-N-1} and \ref{lem:compare-lem-3-4} respectively. Suppose
$\delta_{1}>0$ and $\delta_{2}>0$ are chosen to be small enough
so that they satisfy Proposition \ref{prop:deltas-in-dual-decrease}
and $\frac{1}{2}\delta_{1}^{2}+D\delta_{2}<\min(\bar{\epsilon},\frac{1}{2}\hat{\epsilon}^{2})$,
where 
\[
\begin{array}{c}
D=\underset{(i,r)\in G}{\sum}d(0,H_{i,r})\mbox{, and }G=\{(i,r):c_{i,r}<\infty\}.\end{array}
\]
Denote the right hand side of \eqref{eq:dist-est-formula} by $F(y^{\circ})$.
(Note the $\tilde{y}_{i,r}$ and $x_{0}^{+}$ are derived from $y^{\circ}$.)
We simplify $F(y^{k})$ to be $F_{k}$. As long as $F_{k-1}>\delta_{2}$,
we have 
\[
v(y^{k})\leq v(y^{k-1})-F_{k-1}<v(y^{k-1})-\delta_{2}.
\]
Let $\bar{k}$ be the first $k$ such that $F_{k-1}\leq\delta_{2}$.
By the condition $\frac{1}{2}\delta_{1}^{2}+D\delta_{2}<\bar{\epsilon}$
and how $D$ is defined, we have that $v(y^{\bar{k}-1})<\bar{\epsilon}$.
Lemma \ref{lem:set-N-1} implies that $\delta^{*}(e_{j,j},C_{s(j)})=0$
for all $j>0$, and property (1) implies that $\delta^{*}(e_{j,j'},C_{s(j)})=0$
for all $j$ and $j'$ such that $j\leq j'\leq w'$. This in turn
implies that $\delta^{*}(y_{i}^{+},C_{i})=0$ for all $i\in[1,m]$. 

The local convergence results (Theorems \ref{thm:IP90-lin-conv} and
\ref{thm:DH94-lin-conv}) would ensure local linear convergence. 
\end{proof}

\bibliographystyle{amsalpha}
\bibliography{refs}

\end{document}